\RequirePackage{amsthm}
\documentclass[pdflatex,sn-mathphys-num]{sn-jnl}

\usepackage{graphicx}%
\usepackage{amsmath,amssymb,amsfonts}%
\usepackage{amsthm,mathtools}%
\usepackage{mathrsfs}%
\usepackage{bm}%
\usepackage[dvipsnames]{xcolor}%
\usepackage{booktabs}%
\usepackage{algorithm}%
\usepackage{comment}

\usepackage{newfloat}
\DeclareFloatingEnvironment[name=Algorithm]{algofloat}

\theoremstyle{thmstyleone}%
\newtheorem{theorem}{Theorem}
\newtheorem{proposition}[theorem]{Proposition}%
\newtheorem{lemma}[theorem]{Lemma}%
\newtheorem{corollary}[theorem]{Corollary}%

\theoremstyle{thmstyletwo}%
\newtheorem{example}{Example}%
\newtheorem{remark}{Remark}%

\theoremstyle{thmstylethree}%
\newtheorem{definition}{Definition}%

\DeclareFontFamily{OT1}{pzc}{}
\DeclareFontShape{OT1}{pzc}{m}{it}{<-> s * [1.10] pzcmi7t}{}
\DeclareMathAlphabet{\mathpzc}{OT1}{pzc}{m}{it}

\newcommand{\wid}[1]{\mathrm{w}\left( #1 \right)}

\begin{document}

\title[Relaxation via Separable Estimators]{Relaxation via Separable Estimators: Arithmetic and Implementation}


\author[1]{\fnm{Yanlin} \sur{Zha}}\email{y.zha18@imperial.ac.uk}
\author[2]{\fnm{Mario Eduardo} \sur{Villanueva}}\email{me.villanueva@imtlucca.it}
\author[3]{\fnm{Boris} \sur{Houska}}\email{borish@shanghaitech.edu.cn}
\author*[1]{\fnm{Beno\^it} \sur{Chachuat}}\email{b.chachuat@imperial.ac.uk}


\affil*[1]{\orgdiv{The Sargent Centre for Process Systems Engineering, Department of Chemical Engineering}, \orgname{Imperial College London}, \orgaddress{\city{London}, \country{UK}}}

\affil[2]{\orgdiv{DYSCO Unit}, \orgname{IMT School for Advanced Studies Lucca}, \orgaddress{\city{Lucca}, \country{Italy}}}

\affil[3]{\orgdiv{School of Information Science and Technology}, \orgname{ShanghaiTech University}, \orgaddress{\city{Shanghai}, \country{China}}}


\abstract{
This article presents an arithmetic, called superposition relaxation, for bracketing the graph of a multivariate factorable function on a compact domain between a pair of underestimating and overestimating functions that are both separable. Propagation rules are established for affine and nonlinear composition operations, with a focus on exploiting global monotonicity and convexity properties in the composition. The local convergence properties of this arithmetic are also analyzed in both the pointwise and Hausdorff sense, including conditions under which quadratic pointwise convergence propagates through composition. Parameterizations of the univariate summands in a superposition relaxation either as piecewise-constant or continuous piecewise-linear functions are discussed for a practical implementation. It is shown through numerical case studies that superposition relaxations can be consistently tighter than McCormick relaxations, including for the relaxation of artificial neural networks. But superposition relaxations also incur a higher computational cost than McCormick relaxations. Further investigations are thus warranted as applications in global optimization seek to balance a relaxation's tightness with its computational cost.
}



\maketitle

\section{Introduction}
\label{sec:introduction}

Advances in algorithmic and computing technology over several decades have made it possible to solve nonconvex optimization problems of practical relevance to global optimality within reasonable computational time. Spatial branch-and-bound (sBB) is one of the most prominent complete-search methods and has been implemented in state-of-the-art global solvers such as {\sf BARON} \citep{Tawarmalani2005}, {\sf SCIP} \citep{Vigerske2017} and, more recently, {\sf GUROBI} \citep{gurobi}. It entails the successive partitioning of the decision variable domain in combination with the repeated construction of convex or concave relaxations for the participating objective and constraint functions, so as to exclude subdomains where a global optimum cannot be found \citep{Horst1996,Neumaier2004}. 

The prevalent relaxation approach in sBB hinges on the factorable programming technique \citep{McCormick1976}, where each function is treated as a recursive sum and/or product of univariate functions, introducing auxiliary variables and constraints for these intermediate factors, then relaxing any nonconvex auxiliary constraints, typically by constructing a polyhedral outer-approximation. This approach can guarantee the asymptotic convergence of the sBB algorithm for a broad class of optimization problems \citep{Horst1996,Tawarmalani2002}. However, factorable programming relaxations are subject to the cluster problem \citep{Du1994,Wechsung2014}, where the sBB algorithm can visit a large number of small boxes in the vicinity of a global optimum. Furthermore, the complexity of this polyhedral relaxation increases with the number of nonlinear factors in the objective and constraint functions, and a large number of factors can also cause a dependency problem that weakens the relaxations. A prototypical example of this are optimization problems embedding artificial neural networks, which quickly become intractable to global optimality with the factorable programming technique as the numbers of hidden layers and/or neurons in each layer increase \citep{Schweidtmann2019}. To enhance the convergence rate, it is customary to supplement the relaxations with bounds-tightening techniques and tailored branching strategies \citep{Belotti2009,Gleixner2017}. Significant research has furthermore been devoted to strengthening factorable programming relaxations by detecting convex/concave or componentwise convex/concave subexpressions \citep{Fourer2010,Misener2012}; exploiting special function structures such as multilinear or polynomial subexpressions \citep{Rikun1997,Sherali2012,Zorn2014,Bao2015}; and developing relaxation hierarchies \citep{Sherali1999,Lasserre2001,Parrilo2003,Ahmadi2019}.

Complementing the classical factorable programming technique, factorable reduced-space formulations \cite{Eperly1997,Mitsos2009,Kearfott2013} exploit special structures to reduce the dimensionality of the optimization problem that is exposed to the sBB algorithm, thereby taming the exponential growth of branch-and-bound trees. Any variable for which an explicit relationship with other variables exists through an equality constraint in the original problem can be eliminated by substitution. This strategy can deliver significant speedups, e.g., in optimization with neural networks embedded \citep{Schweidtmann2019} or algorithms with finitely-many steps \citep{Mitsos2009}, flowsheet optimization \citep{Bongartz2019}, and superstructure optimization \citep{Burre2023}. It can also be extended to optimization problems that embed non-factorable functions, including nonlinear algebraic and differential equations \citep{Singer2006,Stuber2015,Scott2015,Villanueva2015a}. 

McCormick relaxations \citep{McCormick1976} have so far been the method of choice in reduced-space formulations and are implemented in global solvers such as {\sf MAiNGO} \citep{MAiNGO} and {\sf EAGO} \citep{Wilhelm2022}. Although fast to compute and quadratically convergent under mild regularity assumptions \citep{Bompadre2012}, McCormick relaxations are not known in closed form, necessitating their repeated calculations at given points. They are furthermore nonsmooth, needing subgradient propagation in practice \citep{Mitsos2009}. And like classical factorable programming relaxations, McCormick relaxations may produce weak relaxations with a large number of intermediate factors, which can induce a cluster problem \citep{Wechsung2014}. Higher-order inclusion techniques such as Taylor and Chebyshev models~\cite{Ratschek1984,Berz1997,Bompadre2013,Rajyaguru2017} provide alternative arithmetics that do not introduce auxiliary variables and describe under- and overestimators in closed form. But the polynomial approximant is generally a nonconvex function that needs to be relaxed as well in a sBB implementation, which may add to the overestimation. Moreover, the computational cost of these estimators grows combinatorially with the degree of the polynomial approximant or the number of variables and the remainder bounds is often overconservative as the variable domain is enlarged.

In this paper, we propose a new type of relaxation, called superposition relaxation, to enclose the graph of a multivariate factorable function between a pair of separable under- and overestimators. A key advantage over McCormick relaxations, where the under- and overestimators are respectively convex and concave by construction, is that the univariate summands in the under- and overestimator of a superposition relaxation can capture global nonlinearities and nonconvexities in the original function. Moreover, these summands can be parameterized in such a way that their extrema can be recovered both efficiently and exactly, for instance as continuous piecewise-linear functions; then, adding up the summand bounds yields the exact range of the superposition relaxation. Due to their separable structure, the computational cost of superposition relaxations is essentially linear in the number of variables. Precursor work includes the interval superposition arithmetic \citep{Zha2018,Su2019}, where the estimators are parameterized by a matrix with interval components and the superposition relies on adding selected components to form an enclosure of the function's graph. The present superposition relaxations are much more general in the sense that they combine univariate summands (as opposed to interval partitions) and do not require the underestimator and overestimator to be symmetric.

A key contribution of the paper is an arithmetic for propagating the superposition relaxations through affine and nonlinear composition operations in a factorable function (Section~\ref{sec:arithmetic}). In particular, the composition rule between a superposition relaxation and an outer univariate function builds upon the concept of ridge function \citep{logan1975,pinkus2015} and exploits global properties of the outer function such as convexity and monotonicity to construct tight under- and overestimators. Related work that exploit separable structures includes the convex relaxation of componentwise convex function \citep{Najman2019}, which was shown to be tighter than its McCormick relaxation counterpart and yields a similar rule as the one derived here for convex ridge function. Our paper takes the next step of combining this underestimator rule with an overestimator rule, then propagating these estimators through multiple operations in factorable expressions. Other contributions of the paper include an analysis of the asymptotic convergence properties of this superposition arithmetic (Section~\ref{sec:convergence}); and two practical implementations of superposition relaxation, where the univariate summands are parameterized as either piecewise constant functions over a fixed partition or continuous piecewise linear functions over an adaptive partition (Section~\ref{sec:parameterization}). Both implementations are available through the open-source library {\sf MC\textsuperscript{++}} \citep{MCPP} and its Python-binding interface {\sf PyMC}, which we use in the numerical case studies (Section~\ref{sec:casestudies}).

\section{Preliminaries and Notation}
\label{sec:background}

Besides common mathematical notation, real vectors are denoted with lowercase bold letters, as in ${\bm x}\in\mathbb{R}^n$. Uppercase letters are used to denote intervals, such as $Z\in\mathbb{IR}$, with $\mathbb{IR}$ the set of intervals in $\mathbb{R}$. $\underline{Z}$ and $\overline{Z}$ refer to the lower and upper bound of $Z$, respectively, and $\wid{Z} \coloneqq \overline{Z} - \underline{Z}$ to its width. Interval vectors, which correspond to the Cartesian product of $n$ intervals, are denoted accordingly with uppercase bold letters, as in ${\bm X}\coloneqq X_1\times\cdots\times X_n\in\mathbb{IR}^{n}$. The diameter of ${\bm X}$ is denoted by $\operatorname{diam}({\bm X}) \coloneqq \sqrt{\sum^{n}_{i=1} \wid{X_{i}}^2}$ \citep{Moore1979}. 

Given a function $f:\mathbb{R}^n\to\mathbb{R}$, we denote the image of the set $\mathcal{X}\subseteq\mathbb{R}^n$ under $f$ by $f(\mathcal{X})\coloneqq \{f({\bm x}) \mid {\bm x}\in\mathcal{X}\}$. Recall that $f(\mathcal{X})$ describes a compact set whenever $\mathcal{X}$ itself is compact and $f$ is continuous on $\mathcal{X}$. In this case, we denote the minimum and maximum of $f$ on $\mathcal{X}$ as $\underline{f}(\mathcal{X}) \coloneqq \min\{f({\bm x}) \mid {\bm x}\in\mathcal{X}\}$ and $\overline{f}(\mathcal{X}) \coloneqq \max\{f({\bm x}) \mid {\bm x}\in\mathcal{X}\}$.

A function $f:\mathbb{R}^n\to\mathbb{R}$ is called \emph{separable} if it can be decomposed as 
\begin{align*}
\forall {\bm x}\in\mathbb{R}^n, \quad f({\bm x}) = \sum_{i=1}^n f_i(x_i) 
\end{align*} 
for some univariate functions $f_i:\mathbb{R}\to\mathbb{R}$. A \emph{ridge function} \citep{logan1975,pinkus2015} is any function $f:\mathbb{R}^n\to\mathbb{R}$ that can be written as the composition of a univariate function $\varphi:\mathbb{R}\to\mathbb{R}$, called profile function, with an affine transformation,
\begin{align*}
\forall {\bm x}\in\mathbb{R}^n, \quad f({\bm x}) \coloneqq \varphi\left({\bm a}^\intercal {\bm x} + b\right)
\end{align*}
for some direction ${\bm a}\in\mathbb{R}^n$ and shift $b\in\mathbb{R}$.

We call \emph{underestimator} and \emph{overestimator} of a function $f:\mathbb{R}^n\to\mathbb{R}$ on the set $\mathcal{X}\subseteq\mathbb{R}^n$ any functions $f^{\rm u},f^{\rm o}:\mathcal{X}\to\mathbb{R}$ such that
\begin{align*}
\forall {\bm x} \in \mathcal{X}, \quad f^{\rm u}({\bm x}) \leq f({\bm x}) \leq f^{\rm o}({\bm x}).
\end{align*}
We further refer to a pair of underestimator and overestimator simply as an {\em estimator}. Specific classes of such estimators are introduced in the following section.

\section{Superposition Relaxation Arithmetic}
\label{sec:arithmetic}

A superposition relaxation is any estimator that brackets the range of an estimated function between an underestimator and an overestimator that are both separable functions. A formal definition follows.

\begin{definition}
\label{def:firstOrderSE}
Let the function $f:\mathcal{X}\to\mathbb{R}$ be defined on the domain $\mathcal{X}\subseteq \mathbb{R}^n$. A \emph{superposition relaxation} of $f$ on ${\bm X}\subseteq \mathcal{X}$, ${\bm X}\in\mathbb{IR}^n$, denoted by $(f^{\rm u},f^{\rm o})_{\bm X}$, is any pair of separable functions $f^{\rm u},f^{\rm o}:{\bm X}\to \mathbb{R}$ that bracket the variations of $f$ on ${\bm X}$,
\begin{align*}
\forall {\bm x}\in{\bm X}, \quad f^{\rm u}({\bm x}) \eqqcolon \sum_{i=1}^n f^{\rm u}_i(x_i) \leq f({\bm x}) \leq f^{\rm o}({\bm x}) \eqqcolon \sum_{i=1}^n f^{\rm o}_i(x_i),
\end{align*}
for some univariate real-valued functions $f^{\rm u}_i,f^{\rm o}_i:X_i\to\mathbb{R}$. We call the separable function $f^{\rm u}$ a \emph{superposition underestimator} and $f^{\rm o}$, a \emph{superposition overestimator}.
\end{definition}

A key property of superposition relaxations lies in their additive complexity in the number of variables $n$, enabled by the separability of $f^{\rm u}$ and $f^{\rm o}$. Another key feature is their ability to describe certain variations and nonconvexities present in the estimated function $f$ through the univariate summands $f^{\rm u}_i$ and $f^{\rm o}_i$. Yet, despite the resulting underestimator $f^{\rm u}$ and overestimator $f^{\rm o}$ being nonconvex functions, the range of a superposition relaxation can still be conveniently recovered by adding the $n$ minima and maxima of the univariate summands $f^{\rm u}_i$ and $f^{\rm o}_i$ on their respective domains,
\begin{align*}
\left[\underline{f}^{\rm u}({\bm X}), \overline{f}^{\rm o}({\bm X})\right] = \left[ \sum_{i=1}^n \underline{f}^{\rm u}_i(X_i), \sum_{i=1}^n \overline{f}^{\rm o}_i(X_i )\right] \supseteq f({\bm X}).
\end{align*}
The choice of a particular parameterization for these univariate underestimators and overestimators should be guided by practical considerations such as the computational burden and ability to calculate tight bounds on the summand range that make it inexpensive to compute the exact range of a superposition relaxation. This discussion is deferred until Section~\ref{sec:parameterization}.

By definition, the estimator $(f,f)_{\bm X}$ provides a trivial superposition relaxation for any univariate or separable function $f$. This is the case, in particular, for any affine functions. The focus in this section, therefore, is on developing an arithmetic to enable the construction of superposition relaxations for nonseparable functions that are factorable---namely, any function which can be represented by a finite combination of binary sums, binary products and outer compositions with a univariate function \citep{McCormick1976}. 

The construction of a superposition relaxation starts with an initialization procedure for the variables participating in the factorable function. Each variable $x_i\in X_i$, $i=1\ldots n$, is readily encoded by a superposition relaxation $(f^{\rm u},f^{\rm o})_{\bm X}$ with the following summands
\begin{align}
\label{eq:initialization}
f^{\rm u}_i(x_i) = f^{\rm o}_i(x_i) = x_i \quad \text{and} \quad f^{\rm u}_{j\neq i}(x_j) = f^{\rm o}_{j\neq i}(x_j) = 0 .
\end{align}

The following results for propagating superposition relaxations through basic addition and scaling operations are straightforward and stated without a proof. The result for a general affine transformation follows through finite composition of these operations.

\begin{proposition}
\label{prop:addition}
Let two functions $g,h:{\bm X}\to\mathbb{R}$ with their respective superposition relaxations $(g^{\rm u},g^{\rm o})_{\bm X}$ and $(h^{\rm u},h^{\rm o})_{\bm X}$ on ${\bm X}$. 
\begin{itemize}
\item A superposition relaxation $(f^{\rm u},f^{\rm o})_{\bm X}$ for $f=g+h$ on ${\bm X}$ is
\begin{align*}
f^{\rm u}_i(x_i) = g^{\rm u}_i(x_i)+h^{\rm u}_i(x_i) \quad \text{and} \quad
f^{\rm o}_i(x_i) = g^{\rm o}_i(x_i)+h^{\rm o}_i(x_i).
\end{align*}
\item A superposition relaxation $(f^{\rm u},f^{\rm o})_{\bm X}$ for $f=a g+b$ with $(a,b)\in\mathbb{R}_+\times\mathbb{R}$ on ${\bm X}$ is
\begin{align*}
f^{\rm u}_i(x_i) = a\: g^{\rm u}_i(x_i)+\frac{b}{n} \quad \text{and} \quad f^{\rm o}_i(x_i) = a\: g^{\rm o}_i(x_i)+\frac{b}{n}.
\end{align*}
\item A superposition relaxation $(f^{\rm u},f^{\rm o})_{\bm X}$ for $f=a g+b$ with $(a,b)\in\mathbb{R}_-\times\mathbb{R}$ on ${\bm X}$ is
\begin{align*}
f^{\rm u}_i(x_i) = a\: g^{\rm o}_i(x_i)+\frac{b}{n} \quad \text{and} \quad f^{\rm o}_i(x_i) = a\: g^{\rm u}_i(x_i)+\frac{b}{n}.
\end{align*}
\end{itemize}
\end{proposition}

The focus in the remainder of this section is on constructing superposition relaxations for product and composition operations. We proceed by first deriving valid superposition relaxations for ridge functions, both without and with monotonicity or convexity assumptions (Sections~\ref{sec:cvxridge} \& \ref{sec:ncvxridge}). These results provide key building blocks for propagating superposition relaxations through composition operations in factorable functions (Section~\ref{sec:composition}). A range reduction operation that intersect a superposition relaxation with a set of a priori bounds is also presented (Section~\ref{sec:tighten}). An asymptotic convergence analysis for the resulting superposition relaxations is detailed in Section~\ref{sec:convergence}.

\subsection{Superposition Relaxation of Convex Ridge Functions}
\label{sec:cvxridge}

The focus in this subsection is on ridge functions of the form
\begin{align}
f({\bm x}) \coloneqq \varphi\left(\sum_{i=1}^n x_i\right),
\label{eq:cvxridge}
\end{align}
where the profile function $\varphi$ is convex. The following results exploit this convexity to construct both a superposition overestimator (Proposition~\ref{prop:cvxridge_over}) and a superposition underestimator (Proposition~\ref{prop:cvxridge_under}) for such functions. Similar results for a concave profile function $\varphi$ follow by symmetry.

\begin{proposition}
\label{prop:cvxridge_over}
Let ${\bm X} \in \mathbb{IR}^n$ with $\wid{X_i}>0$ for some $i=1,\ldots,n$. Let the function $\varphi:Z\to \mathbb R$ be defined on a convex set $Z\subseteq\mathbb{R}$ such that $\sum_{i=1}^n x_i \in Z$ for all ${\bm x}\in{\bm X}$ and be convex on $Z$. Then, for all ${\bm x}\in{\bm X}$, we have
\begin{align}
\varphi\left(\sum_{i=1}^n x_i\right)
\leq \sum_{i\in \mathcal{P}}\theta_i\varphi\left(\frac{x_i-\underline{X}_i}{\theta_i}+\underline{\sigma}\right) 
= \sum_{i\in \mathcal{P}}\theta_i\varphi\left(\frac{x_i-\overline{X}_i}{\theta_i}+\overline{\sigma}\right), \label{eq:cvxridge_over}
\end{align}
with $\theta_i \coloneqq \frac{\wid{X_i}}{\sum_{j \in \mathcal{P}}\wid{X_j}}$ for all $i\in \mathcal{P} \coloneqq \left\{j=1,\ldots,n \mid \wid{X_j} > 0 \right\}$, $\underline{\sigma} \coloneqq \sum_{i=1}^n\underline{X}_i$, and $\overline{\sigma} \coloneqq \sum_{i=1}^n\overline{X}_i$.
\end{proposition}

\begin{proof}
We start by noting that, for all $i\in\mathcal{P}$ and all $x_i\in X_i$, we have
\begin{align*}
\overline{\sigma}-\underline{\sigma} = \sum_{j=1}^n \wid{X_j} = \frac{\wid{X_i}}{\theta_i} = \frac{x_i-\underline{X}_i}{\theta_i} - \frac{x_i-\overline{X}_i}{\theta_i},
\end{align*}
which implies the equality in \eqref{eq:cvxridge_over}. Next, by convexity of $\varphi$ on $Z$, Jensen's inequality gives
\begin{align*}
\forall z_i\in Z,\ i\in\mathcal{P}, \quad \varphi\left(\sum_{i\in\mathcal{P}} \alpha_i z_i\right) \leq \sum_{i\in\mathcal{P}} \alpha_i \varphi\left( z_i\right),
\end{align*}
for any $\alpha_i\geq 0$ with $\sum_{i\in\mathcal{P}} \alpha_i=1$. The result follows by choosing $\alpha_i=\theta_i$ and $z_i = \frac{x_i-\underline{X}_i}{\theta_i}+\underline{\sigma}$ for $i\in\mathcal{P}$, and noting that $\sum_{i\in\mathcal{P}} \theta_i z_i = \sum_{i=1}^n x_i$ whenever $x_i=\underline{X}_i$ for $i\notin\mathcal{P}$. 
\end{proof}

\begin{remark}
\label{rmk:cvxridge_over}
The particular weight selection $\theta_i$ in Proposition~\ref{prop:cvxridge_over} comes with the desirable property that, for every $i=1,\ldots,n$, we have
\begin{align*}
\max_{x_i\in X_i} \varphi\left(\frac{x_i-\underline{X}_i}{\theta_i}+\underline{\sigma}\right) = \max_{z\in[\underline{\sigma},\overline{\sigma}]} \varphi(z) \ = \max\{\varphi(\underline{\sigma}),\varphi(\overline{\sigma})\}.
\end{align*}
Therefore, the ridge function \eqref{eq:cvxridge} and its overestimator share the same set of maximizers on $\bm X$. More generally and irrespective of the convexity of $\varphi$, \eqref{eq:cvxridge_over} holds with equality along the diagonal line segment connecting the lower-corner $(\underline{X}_1,\ldots,\underline{X}_n)$ to the upper-corner $(\overline{X}_1,\ldots,\overline{X}_n)$. This is because any point ${\bm x}$ on this line segment is such that $x_i = \underline{X}_i + \eta(\overline{X}_i-\underline{X}_i)=\underline{X}_i+\eta\theta_i(\overline{\sigma}-\underline{\sigma})$ for some $\eta\in[0,1]$, and therefore we have
\begin{align*}
\sum_{i\in\mathcal{P}}\theta_i\varphi\left(\frac{x-\underline{X}_i}{\theta_i}+\underline{\sigma}\right) 
& = \sum_{i \in P}\theta_i\varphi\left(\eta(\overline{\sigma}-\underline{\sigma})+\underline{\sigma}\right) =  \varphi\left(\eta(\overline{\sigma}-\underline{\sigma})+\underline{\sigma}\right) = \varphi\left(\sum_{i=1}^n x_i\right).
\end{align*}
Notice also that the equality in \eqref{eq:cvxridge_over} may only hold for the particular weight selection $\theta_i \coloneqq \frac{\wid{X_i}}{\sum_{j \in \mathcal{P}}\wid{X_j}}$, although the two separate inequalities would of course remain valid for any convex combination of the weights based on Jensen's inequality.
\end{remark}

\begin{proposition}
\label{prop:cvxridge_under}
Let ${\bm X} \in \mathbb{IR}^n$ with $\wid{X_i}>0$ for some $i=1,\ldots,n$. Let the function $\varphi:Z\to \mathbb R$ be defined on a convex set $Z\subseteq\mathbb{R}$ such that $\sum_{i=1}^n x_i \in Z$ for all ${\bm x}\in{\bm X}$ and be convex on $Z$. Then, for all ${\bm x}\in{\bm X}$, we have
\begin{align}
\varphi\left(\sum_{i=1}^n x_i\right) \geq \sum_{i=1}^n\varphi\left(x_i-\underline{X}_i+\underline{\sigma}\right) - (n-1)\varphi(\underline{\sigma}) = \sum_{i\in\mathcal{P}}\left(\varphi\left(x_i-\underline{X}_i+\underline{\sigma}\right) - \left(1-\theta_i\right)\varphi(\underline{\sigma})\right), \label{eq:cvxridge_under1}\\
\varphi\left(\sum_{i=1}^n x_i\right) \geq \sum_{i=1}^n\varphi\left(x_i-\overline{X}_i+\overline{\sigma}\right) - (n-1)\varphi(\overline{\sigma}) = \sum_{i\in\mathcal{P}} \left(\varphi\left(x_i-\overline{X}_i+\overline{\sigma}\right) - (1-\theta_i)\varphi(\overline{\sigma})\right),\label{eq:cvxridge_under2}
\end{align}
with $\theta_i \coloneqq \frac{\wid{X_i}}{\sum_{j \in \mathcal{P}}\wid{X_j}}$ for all $i\in \mathcal{P} \coloneqq \left\{j=1,\ldots,n \mid \wid{X_j} > 0 \right\}$, $\underline{\sigma} \coloneqq \sum_{i=1}^n\underline{X}_i$, and $\overline{\sigma} \coloneqq \sum_{i=1}^n\overline{X}_i$.
\end{proposition}

The following lemma is instrumental to prove Proposition~\ref{prop:cvxridge_under}.

\begin{lemma}
\label{lem:cvxridge_under}
Let $\varphi:Z\to \mathbb R$ be convex on the interval $Z\in\mathbb{IR}$ and either set $\omega = \underline{Z}$ or $\omega = \overline{Z}$. Then, for all $\delta_1,\delta_2\in \mathbb R$ such that $\underline{Z} \leq \omega+\delta_1+\delta_2 \leq \overline{Z}$ and $\delta_1\delta_2\geq 0$, we have 
\begin{align}
\label{eq:cvxridge_under0}
\varphi(\omega+\delta_1+\delta_2)+\varphi(\omega)-\varphi(\omega+\delta_1)-\varphi(\omega+\delta_2) \geq 0 .
\end{align}
\end{lemma}

\begin{proof}
Notice that \eqref{eq:cvxridge_under0} holds with equality whenever $\delta_1=0$ or $\delta_2=0$. In the other case that $\delta_1\delta_2>0$, we have $0<\frac{\delta_1}{\delta_1+\delta_2},\frac{\delta_2}{\delta_1+\delta_2}<1$ in both scenarios $\omega = \underline{Z}$ and $\omega = \overline{Z}$, and since $\varphi$ is convex on $Z$,
\begin{align*}
\varphi(\omega+\delta_1) =\ & \varphi\left( \frac{\delta_2}{\delta_1+\delta_2}\omega+\frac{\delta_1}{\delta_1+\delta_2}(\omega+\delta_1+\delta_2) \right)\\
\leq\ & \frac{\delta_2}{\delta_1+\delta_2}\varphi(\omega) + \frac{\delta_1}{\delta_1+\delta_2}\varphi(\omega+\delta_1+\delta_2),\\
\varphi(\omega+\delta_2) =\ & \varphi\left( \frac{\delta_1}{\delta_1+\delta_2}\omega+\frac{\delta_2}{\delta_1+\delta_2}(\omega+\delta_1+\delta_2) \right)\\
\leq\ & \frac{\delta_1}{\delta_1+\delta_2}\varphi(\omega) + \frac{\delta_2}{\delta_1+\delta_2}\varphi(\omega+\delta_1+\delta_2).
\end{align*}
The result in \eqref{eq:cvxridge_under0} follows by adding the previous two inequalities.
\end{proof}

\begin{proof}[Proof of Proposition~\ref{prop:cvxridge_under}]
We start by rewriting the inequality in \eqref{eq:cvxridge_under1} as
\begin{align}
\varphi\left(\underline{\sigma}+\sum_{i=1}^k \delta_i\right) \geq \sum_{i=1}^k\varphi\left(\delta_i+\underline{\sigma}\right) - (k-1)\:\varphi(\underline{\sigma}), \label{eq:cvxridge_under1b}
\end{align}
with $\delta_i \coloneqq x_i-\underline{X}_i$. We proceed by induction reasoning to establish the validity of inequality~\eqref{eq:cvxridge_under1b} for $k=n$. Note first that \eqref{eq:cvxridge_under1b} holds with equality for $k=1$. Then, assuming that inequality~\eqref{eq:cvxridge_under1b} holds for $k\geq 1$, we have
\begin{align*}
& \varphi\left(\underline{\sigma}+\sum_{i=1}^{k+1} \delta_i\right) - \sum_{i=1}^{k+1}\varphi\left(\underline{\sigma}+\delta_i\right) + k\varphi(\underline{\sigma})\\
=\ & \underbrace{\varphi\left(\underline{\sigma} + \sum_{i=1}^{k+1} \delta_i\right) + \varphi\left(\underline{\sigma}\right) - \varphi\left(\underline{\sigma} + \sum_{i=1}^{k} \delta_i\right) - \varphi\left(\underline{\sigma} + \delta_{k+1}\right)}_{\displaystyle\geq 0 \text{  by Lemma~\ref{lem:cvxridge_under}}}\\
& + \underbrace{\varphi\left(\underline{\sigma} + \sum_{i=1}^{k} \delta_i\right) + (k-1)\: \varphi\left(\underline{\sigma}\right)  - \sum_{i=1}^{k}\varphi\left(\underline{\sigma}+\delta_i\right)}_{\displaystyle\geq 0 \text{  since \eqref{eq:cvxridge_under1b} holds for $k$}}.
\end{align*}
Therefore, inequality~\eqref{eq:cvxridge_under1b} also holds for $k+1$, which proves the result for $k=n$ and the validity of the inequality in \eqref{eq:cvxridge_under1}. The equality in \eqref{eq:cvxridge_under1} follows by noting that $\sum_{i\notin\mathcal{P}} \varphi(x_i-\underline{X}_i+\underline{\sigma}) = (n-|\mathcal{P}|)\,\varphi(\underline{\sigma})$ and $\sum_{i\in\mathcal{P}} \theta_i=1$. The proof of \eqref{eq:cvxridge_under2} proceeds analogously.
\end{proof}

\begin{remark}
For a continuously-differentiable profile function $\varphi$, the result in Proposition~\ref{prop:cvxridge_under} may also be derived as a special case of Theorem~1 and Remark~2 in \citep{Najman2019}. This is because the ridge function \eqref{eq:cvxridge} is componentwise convex as $x_i\mapsto f(x_1^\ast,\ldots,x_{i-1}^\ast,x_i,x_{i+1}^\ast,\ldots,x_n^\ast)$ is itself convex for any fixed $x_j^\ast\in X_j$, $j\neq i$. However, unlike in \citep{Najman2019}, the proof of Proposition~\ref{prop:cvxridge_under} does not impose any first-order continuity condition on $\varphi$ on top of convexity.
\end{remark}

\begin{remark}
\label{rmk:cvxridge_under}
When $\varphi$ is nondecreasing, the superposition underestimator in \eqref{eq:cvxridge_under1} shares the same minimizer as the ridge function \eqref{eq:cvxridge} on $\bm X$ at the lower corner $(\underline{X}_1,\ldots,\underline{X}_n)$. More generally and irrespective of the convexity of $\varphi$, \eqref{eq:cvxridge_under1} holds with equality along each edge of the interval box ${\bm X}$ with $x_i\in X_i$ and $x_{j\neq i}=\underline{X}_j$,
\begin{align*}
\sum_{j=1}^n\varphi\left(x_j-\underline{X}_j+\underline{\sigma}\right) - (n-1)\varphi(\underline{\sigma}) = \varphi\left(x_i-\underline{X}_i+\underline{\sigma}\right) = \varphi\left(x_i+\sum_{\substack{j=1,\\j\neq i}}^n \underline{X}_j\right).
\end{align*}
The ridge function \eqref{eq:cvxridge} and its superposition underestimator in \eqref{eq:cvxridge_under1} and  thus match along $n$ edges of ${\bm X}$, which define a cone pointing at the lower corner $(\underline{X}_1,\ldots,\underline{X}_n)$. Symmetrically when $\varphi$ is nonincreasing, the superposition underestimator in \eqref{eq:cvxridge_under2} shares the same minimizer as \eqref{eq:cvxridge} at the upper corner $(\overline{X}_1,\ldots,\overline{X}_n)$ and they match along those $n$ edges of ${\bm X}$ defining a cone pointing at the upper corner $(\overline{X}_1,\ldots,\overline{X}_n)$.
\end{remark}

It follows from Propositions~\ref{prop:cvxridge_over} and \ref{prop:cvxridge_under} that a superposition relaxation $(f^{\rm u},f^{\rm o})_{\bm X}$ for the ridge function \eqref{eq:cvxridge} with a convex profile function $\varphi$ is obtained with the following summands for all $i\in\mathcal{P}$
\begin{align}
f^{\rm u}_i(x_{i}) \coloneqq\ & \varphi\left(x_{i}-\underline{X}_i+\underline{\sigma}\right) - \left(1-\theta_i\right)\varphi(\underline{\sigma})\label{eq:monotonic_underest}\\
f^{\rm o}_i(x_{i}) \coloneqq\ & \theta_i\varphi\left(\frac{x_{i}-\underline{X}_i}{\theta_i}+\underline{\sigma}\right), \label{eq:monotonic_overest}
\end{align}
and $f^{\rm u}_i(x_{i})=f^{\rm o}_i(x_{i})=0$ for $i\notin\mathcal{P}$. In particular, the order between each pair of underestimator and overestimator summands $(f^{\rm u}_i,f^{\rm o}_i)_{1\leq i\leq n}$ is preserved due to the convexity of $\varphi$,
\begin{align*}
\forall x_{i}\in X_i, \quad f^{\rm u}_i(x_{i}) \leq f^{\rm o}_i(x_{i}),
\end{align*}
with equality at $f^{\rm u}_i(\underline{X}_i) = f^{\rm o}_i(\underline{X}_i)=\theta_i\varphi(\underline{\sigma})$. Notice also that a superposition underestimator for a convex (respectively, concave) profile function $\varphi$ consists of convex (respectively, concave) summands only; and since the perspective $\varphi_\theta:z\mapsto \theta\varphi(z/\theta)$ is again convex (respectively, concave) for any given $\theta>0$ \citep{Hiriart2001}, a superposition overestimation too consists of convex (respectively, concave) summands only. 

Regarding extremum-tracking properties, it readily follows from Remark~\ref{rmk:cvxridge_over} that the superposition overestimator $f^{\rm o}$ in \eqref{eq:monotonic_overest} satisfies $\overline{f}^{\rm o}({\bm X}) = \overline{f}({\bm X})$; while it follows from Remark~\ref{rmk:cvxridge_under}, the superposition underestimator $f^{\rm u}$ in \eqref{eq:monotonic_underest} satisfies $\underline{f}^{\rm u}({\bm X}) = \underline{f}({\bm X})$ either when $\varphi$ is nondecreasing and the summands $f^{\rm u}_i$ are constructed via \eqref{eq:cvxridge_under1}, or when $\varphi$ is nonincreasing and the summands $f^{\rm u}_i$ are constructed via \eqref{eq:cvxridge_under2}. 

However, Remark~\ref{rmk:cvxridge_under} also suggests that a nonzero estimation gap $\underline{f}^{\rm u}({\bm X}) < \underline{f}({\bm X})$ should be expected in the nonmonotonic case. This occurs specifically when the minimum of $\varphi$ on $[\underline{\sigma},\overline{\sigma}]$ does not belong to any of the edges originating from the lower or upper corner of the domain $\bm X$. 
A simple remedy approach entails decomposing any nonmonotonic function $\varphi$ into the sum of a convex nonincreasing term, a convex nondecreasing term, and a constant term, as follows
\begin{equation}
\label{eq:monotonic_decomp}
\varphi(z) = \varphi(\min\{z,\sigma^\ast\}) + \varphi(\max\{z,\sigma^\ast\}) - \varphi(\sigma^\ast),
\end{equation}
then relaxing the convex nondecreasing and convex nonincreasing terms based on the inequalities \eqref{eq:cvxridge_under1} and \eqref{eq:cvxridge_under2}, respectively. This approach yields the following summands in the superpositon underestimator $f^{\rm u}$ of a nonmonotonic convex function $\varphi$,
\begin{align}
\label{eq:nonmonotonic_underest}
f^{\rm u}_i(x_{i}) \coloneqq\ & \varphi\left(\min\left\{x_{i}-\overline{X}_i+\overline{\sigma},\sigma^\ast\right\}\right) + \varphi\left(\max\left\{x_{i}-\underline{X}_i+\underline{\sigma},\sigma^\ast\right\}\right) - (2-\theta_i)\varphi(\sigma^\ast),
\end{align}
for $i\in\mathcal{P}$, and $f_i^{\rm u}(x_{i})=0$ otherwise. In particular, $f^{\rm u}$ with the summands \eqref{eq:nonmonotonic_underest} is equal to the minimum $\varphi(\sigma^\ast)$ of the ridge function \eqref{eq:cvxridge} at any ${\bm x}\in{\bm X}$ with $x_i\in[\sigma^\ast+\overline{X}_i-\overline{\sigma},\sigma^\ast+\underline{X}_i-\underline{\sigma}]$. It should also be noted that the composition of $\varphi$ with $\min$ and $\max$ functions as in \eqref{eq:monotonic_decomp} does not impede the quadratic convergence of the underestimation error around $\sigma^\ast$ as long as $\varphi$ is differentiable and its derivative is Lipschitz continuous on $[\underline{\sigma},\overline{\sigma}]$---see Section~\ref{sec:convergence} for further details.

\begin{example}
\label{exm:sqr}
We illustrate the previous developments for the ridge function $f({\bm x})=(x_1+x_2)^2$. On the domain ${\bm X}=[0.2,1]^2$, the term $(x_1+x_2)$ varies in the range $[\underline{\sigma},\overline{\sigma}] = [0.4,2]$, where the function $\varphi:z\mapsto z^2$ is convex and nondecreasing. A superposition relaxation that tracks the range of $(x_1+x_2)^2$, therefore, consists of the following summands:
\begin{align*}
f_1^{\rm u}(\xi) =\ & f_2^{\rm u}(\xi) = (\xi+0.2)^2-0.08\\
f_1^{\rm o}(\xi) =\ & f_2^{\rm o}(\xi) = 2\xi^2.
\end{align*}
On the wider domain ${\bm X}=[-1,2]^2$, the term $(x_1+x_2)$ now varies in the range $[\underline{\sigma},\overline{\sigma}] = [-2,4]$, where the function $\varphi:z\mapsto z^2$ is convex and reaches its minimum at $\sigma^\ast=0$. A superposition relaxation that again tracks the range of $(x_1+x_2)^2$, consists of the following summands:
\begin{align*}
f_1^{\rm u}(\xi) =\ & f_2^{\rm u}(\xi) = \max\{\xi-1,0\}^2\\
f_1^{\rm o}(\xi) =\ & f_2^{\rm o}(\xi) = 2\xi^2.
\end{align*}
The superposition overestimation thus remains the same on the wider domain, but not the superposition underestimator. On both domains, however, the superposition relaxation correctly tracks the extrema of the ridge function as illustrated in Figure~\ref{fig:example1}.
\end{example}

\begin{figure}[htb]
\centering
\includegraphics[width=0.48\linewidth]{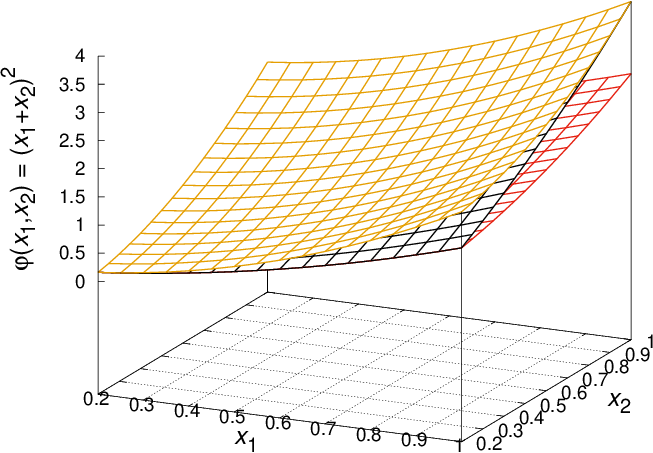}%
\hfill%
\includegraphics[width=0.48\linewidth]{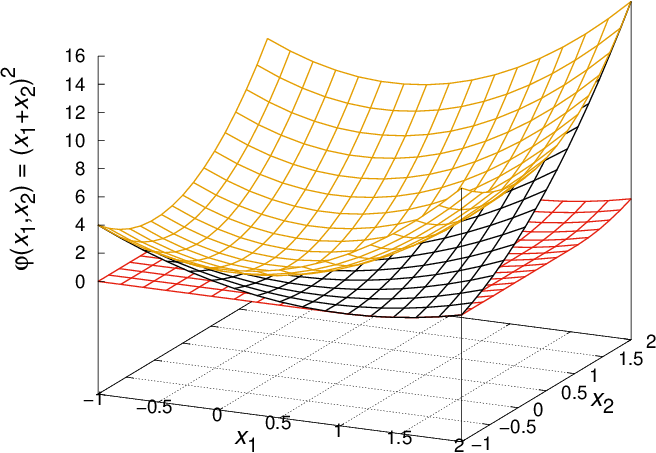}
\caption{Superposition relaxations of the ridge function $f({\bm x})=(x_1+x_2)^2$ on the domains ${\bm X}=[0.2,1]^2$ (left plot) and ${\bm X}=[-1,2]^2$ (right plot). The ridge function is shown in back lines, the superposition underestimator in red lines, and the superposition overestimator in orange lines. }
\label{fig:example1}
\end{figure}

Regarding the optimality of the superposition relaxations of convex ridge functions finally, we conduct a numerical investigation in Appendix~\ref{app:optimality}. The results of this investigation suggest that the summands \eqref{eq:monotonic_overest} may provide an optimal overestimator in the $L_1$ sense, among all possible separable
overestimators of the convex increasing ridge function \eqref{eq:cvxridge}. For a nondecreasing profile function $\varphi$, the summands \eqref{eq:monotonic_underest} may also provide an optimal underestimator both in the $L_1$ and $L_2$ sense, among all possible separable underestimators of the convex increasing ridge function \eqref{eq:cvxridge} that track its exact minimum. However, this optimality may no longer hold for a nonmonotonic profile function $\varphi$ when $f^{\rm u}$ is constructed with the summands \eqref{eq:nonmonotonic_underest}. A more formal analysis of these optimality properties should be conducted as part of future research.

\subsection{Superposition Relaxation of Nonconvex Ridge Functions}
\label{sec:ncvxridge}

With capability to construct superposition relaxations for the ridge function \eqref{eq:cvxridge} with a convex (or concave) profile function $\varphi$, we now shift focus to the case of a nonconvex profile function. One approach to handling a nonconvex $\varphi$ is through its decomposition as a difference of convex (DC) terms, then adding the superposition relaxations of both terms using the results in Section~\ref{sec:cvxridge}. In particular, recall that any twice-continuously differentiable function is DC decomposable and every continuous function can be approximated by a DC function with any desired precision \citep{hartman1959,horst1999}, albeit these decompositions are not unique.

For the nonconvex profile functions of interest, including $\varphi:z\mapsto z^{2p+1}$ with $p\geq 1$, $\varphi:z\mapsto \tanh(z)$ or $\varphi:z\mapsto \cos(z)$, the inflection points on the interval $[\underline{\sigma},\overline{\sigma}]$ are either known or can be computed easily to any finite precision. Formally, any twice-continuously differentiable function $\varphi$ on $[\underline{\sigma},\overline{\sigma}]$ with $m\geq 1$ inflection points, $\underline{\sigma}<\sigma_1<\cdots<\sigma_m<\overline{\sigma}$ can be decomposed as
\begin{equation}
\label{eq:ncvxdecomp}
\varphi(z) = \sum_{k=0}^m\varphi_{k}(z)
\end{equation}
with
\begin{align*}
\varphi_{0}(z) \coloneqq \left\{ \begin{array}{ll}
  \varphi(z) & \text{if $z \leq \sigma_1$}\\
  \varphi'(\sigma_1)\:(z- \sigma_1) + \varphi(\sigma_{1}) & \text{if $z > \sigma_1$,} 
\end{array}\right.
\end{align*}
for $k=1,\ldots, m-1$,
\begin{align*}
\varphi_{k}(z) \coloneqq \left\{ \begin{array}{ll}
  0 & \text{if $z < \sigma_k$}\\
  \varphi(z) - \varphi_{k-1}(z) & \text{if $\sigma_k\leq z\leq \sigma_{k+1}$}\\
  \varphi'(\sigma_{k+1})\:(z-\sigma_{k+1}) + \varphi(\sigma_{k+1}) - \varphi_{k-1}(z) & \text{if $z > \sigma_{k+1}$,}
\end{array}\right.
\end{align*}
and
\begin{align*}
\varphi_{m}(z) \coloneqq \left\{ \begin{array}{ll}
  0 & \text{if $z < \sigma_m$}\\
  \varphi(z) - \varphi_{m-1}(z) & \text{if $z\geq \sigma_m$,}
  \end{array}\right.
\end{align*}
where $\varphi'(\sigma_k)$ denotes the first derivative of $\varphi$ at $\sigma_k$. In particular, the factors $\varphi_k$ are twice-continuously differentiable on $[\underline{\sigma},\overline{\sigma}]$ by construction. A superposition relaxation of \eqref{eq:ncvxdecomp} is then obtained by simply adding superposition relaxations of the factors $\varphi_k$. Notice that $\varphi_1,\ldots,\varphi_m$ are all increasing convex or decreasing concave functions, while $\varphi_0$ will be nonmonotonic when $\varphi$ itself happens nonmonotonic on $[\underline{\sigma},\sigma_1]$. A similar decomposition is also possible when $\varphi$ is nonsmooth---but still Lipschitz-continuous---at any inflection point $\sigma_k$, in which case $\varphi'(\sigma_k)$ is taken as the smallest (respectively, largest) subgradient in the subdifferential of $\varphi$ at $\sigma_k^-$ if $\varphi$ is convex (respectively, concave) on $[\sigma_{k-1},\sigma_k]$. In general, summing the superposition relaxations for factors $\varphi_k$ in the decomposition \eqref{eq:ncvxdecomp} removes any guarantees that the resulting superposition relaxation will track the extrema of the estimated ridge function.

\begin{example}
\label{exm:tanh}
We consider the ridge function $f({\bm x})=\tanh(x_1+x_2)$ on the domain ${\bm X}=[-0.5,1]^2$. The profile function $\varphi:z\mapsto \tanh(z)$ is nonconvex on the corresponding range $[\underline{\sigma},\overline{\sigma}] = [-1,2]$, with a unique inflection point at $\sigma_1=0$. We proceed by decomposing $\varphi$ as in \eqref{eq:ncvxdecomp},
\begin{align*}
\varphi(z) = \underbrace{\max\{\tanh(z),z\}}_{\displaystyle \eqqcolon \varphi_0(z)} + \underbrace{\min\{\tanh(z)-z,0\}}_{\displaystyle \eqqcolon \varphi_1(z)},
\end{align*}
then summing superposition relaxations for the two terms, which are respectively convex increasing and concave increasing on $[-1,2]$:
\begin{align*}
f_1^{\rm u}(\xi) =\ & f_2^{\rm u}(\xi) = \max\{\tanh(\xi-0.5),\xi-0.5\} + \frac{1}{2}\min\{\tanh(2\xi)-2\xi,0\} + \frac{1}{2}\tanh(1)\\
f_1^{\rm o}(\xi) =\ & f_2^{\rm o}(\xi) = \frac{1}{2}\max\{\tanh(2\xi),2\xi\} + \min\{\tanh(\xi-0.5)-\xi+0.5,0\}.
\end{align*}
Notice, in particular, how the superposition underestimator on the left plot of Figure~\ref{fig:example2} correctly tracks the minimum of the ridge function at $x_1=x_2=-0.5$, while the superposition overestimator becomes conservative. The latter is caused by a dependency problem, despite the superposition overestimators of $\varphi_0(x_1+x_2)$ and $\varphi_1(x_1+x_2)$ both tracking the corresponding maxima exactly.
\end{example}

\begin{figure}[htb]
\centering
\includegraphics[width=0.48\linewidth]{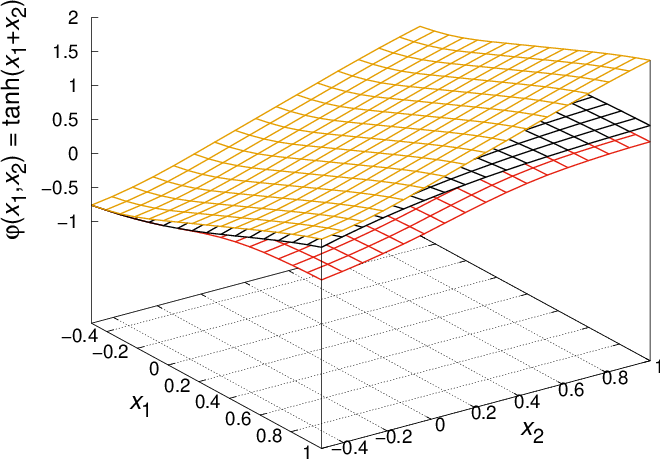}%
\hfill%
\includegraphics[width=0.48\linewidth]{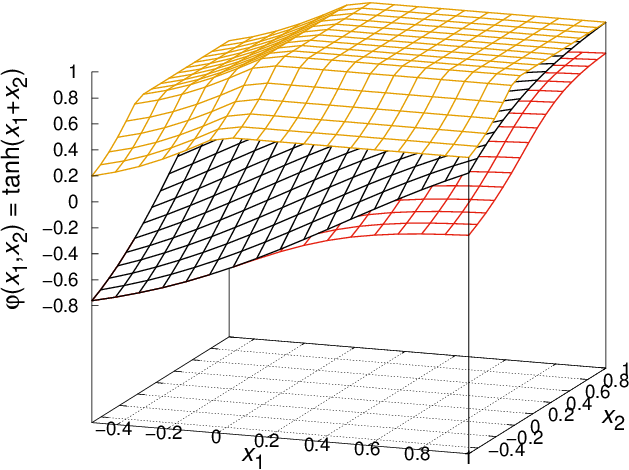}
\caption{Superposition relaxations of the ridge function $f({\bm x})=\tanh(x_1+x_2)$ on the domain ${\bm X}=[-0.5,1]^2$ (left plot) and after intersection with the exact range $[\tanh(-1),\tanh(2)]$ of $f$ (right plot). The ridge function is shown in back lines, the superposition underestimator in red lines, and the superposition overestimator in orange lines. }
\label{fig:example2}
\end{figure}

\subsection{Composition of Superposition Relaxations}
\label{sec:composition}

We exploit the previous results on the construction of superposition relaxations for ridge functions to propagate a superposition relaxation through composition with a univariate function. The following theorem enables this propagation for convex monotonic univariate terms. A similar result holds for concave monotonic terms by symmetry. 

\begin{theorem}
\label{thm:cvxcomposition}
Let $(g^{\rm u},g^{\rm o})_{\bm X}$ be a superposition relaxation of the function $g:{\bm X}\to\mathbb{R}$ on ${\bm X}\in\mathbb{IR}^n$, and define the sets $\mathcal{P}^{\rm u} \coloneqq \{i=1,\ldots, n \mid \wid{g^{\rm u}_i(X_i)}>0 \}$ and $\mathcal{P}^{\rm o} \coloneqq \{i=1,\ldots, n \mid \wid{g^{\rm o}_i(X_i)}>0 \}$. Let the function $\varphi:Z\to\mathbb{R}$ be defined on $Z\supseteq [\underline{g}^{\rm u}({\bm X}),\overline{g}^{\rm o}({\bm X})]$ and assume that $\varphi$ is both convex and monotonic. If $\mathcal{P}^u\neq\emptyset$ and $\mathcal{P}^o\neq\emptyset$, a superposition relaxation $(f^{\rm u},f^{\rm o})_{\bm X}$ of the composite function $f=\varphi\circ g$ on ${\bm X}$ is obtained with the following non-zero summands,
\begin{itemize}
\item case of nondecreasing $\varphi$:
\begin{align*}
f^{\rm u}_i(x_i) =\ & \varphi\left(g^{\rm u}_i(x_i)-\underline{g}^{\rm u}_i(X_i)+\underline{g}^{\rm u}({\bm X})\right) - \left(1-\theta^{\rm u}_i\right)\,\varphi\left(\underline{g}^{\rm u}({\bm X})\right), \quad i\in\mathcal{P}^{\rm u}\\
f^{\rm o}_i(x_i) =\ & \theta^{\rm o}_i \varphi\left(\dfrac{g^{\rm o}_i(x_i)-\overline{g}^{\rm o}_i(X_i)}{\theta^{\rm o}_i} + \overline{g}^{\rm o}({\bm X})\right), \quad i\in\mathcal{P}^{\rm o}
\end{align*}
\item case of nonincreasing $\varphi$:
\begin{align*}
f^{\rm u}_i(x_i) =\ & \varphi\left(g^{\rm o}_i(x_i)-\overline{g}^{\rm o}_i(X_i)+\overline{g}^{\rm o}({\bm X})\right) - \left(1-\theta^{\rm o}_i\right)\,\varphi\left(\overline{g}^{\rm o}({\bm X})\right), \quad i\in\mathcal{P}^{\rm o}\\
f^{\rm o}_i(x_i) =\ & \theta^{\rm u}_i \varphi\left(\dfrac{g^{\rm u}_i(x_i)-\underline{g}^{\rm u}_i(X_i)}{\theta^{\rm u}_i} + \underline{g}^{\rm u}({\bm X})\right), \quad i\in\mathcal{P}^{\rm u}
\end{align*}
\end{itemize}
with $\theta^{\rm u}_i \coloneqq \dfrac{\wid{g^{\rm u}_i(X_i)}}{\wid{g^{\rm u}(\bm X)}}$, $\theta^{\rm o}_i \coloneqq \dfrac{\wid{g^{\rm o}_i(X_i)}}{\wid{g^{\rm o}(\bm X)}}$, and the rest of the summands equal to zero.
\end{theorem}

\begin{proof}
For $(f^{\rm u},f^{\rm o})_{\bm X}$ to define a valid superposition relaxation of $\varphi\circ g$ on ${\bm X}$, we must have
\begin{align}
\label{eq:unary_ieq}
\forall {\bm x}\in{\bm X}, \quad \sum_{i=1}^n f^{\rm u}_i(x_i) \leq \varphi\circ g({\bm x}) \leq \sum_{i=1}^n f^{\rm o}_i(x_i),
\end{align}
We prove the validity of this inequality in the nondecreasing case and note that the result in the nonincreasing case follows by symmetry. Since $(g^{\rm u},g^{\rm o})_{\bm X}$ is a superposition relaxation of $g$ on ${\bm X}$ and $\varphi$ is nondecreasing, we have
\begin{align*}
\varphi\left(\sum_{i=1}^n g^{\rm u}_i(x_i)\right) \leq \varphi\circ g({\bm x}) \leq \varphi\left(\sum_{i=1}^n g^{\rm o}_i(x_i)\right),
\end{align*}
for all $\bm x \in {\bm X}$. Then, by Proposition~\ref{prop:cvxridge_under} and by definition of the $f_i^{\rm u}$ summands,
\begin{align*}
\varphi\circ g({\bm x}) \geq\ & \sum_{i\in\mathcal{P}^{\rm u}} \left( \varphi\left(g^{\rm u}_i(x_i)-\underline{g}^{\rm u}(X_i)+\underline{g}^{\rm u}({\bm X})\right) - (1-\theta_i^{\rm u})\:\varphi\left(\underline{g}^{\rm u}({\bm X})\right) \right)
= \sum_{i=1}^n f^{\rm u}_i(x_i),
\end{align*}
which proves the left-hand side in \eqref{eq:unary_ieq}. Similarly, by Proposition~\ref{prop:cvxridge_over} and by definition of the $f_i^{\rm o}$ summands
\begin{align*}
\varphi\circ g({\bm x}) \leq\ & \sum_{i\in\mathcal{P}^{\rm o}} \theta^{\rm o}_i\varphi\left(\dfrac{g^{\rm o}_i(x_i)-\overline{g}^{\rm o}_i(X_i)}{\theta^{\rm o}_i} + \overline{g}^{\rm o}({\bm X})\right) = \sum_{i=1}^n f^{\rm o}_i(x_i),
\end{align*}
which proves the right-hand side in \eqref{eq:unary_ieq}.
\end{proof}

The composition rule in Theorem~\ref{thm:cvxcomposition} is restricted to monotonic, convex or concave, univariate outer functions, which includes operations such as $z\mapsto \sqrt{z}$, $z\mapsto \frac{1}{z}$, $z\mapsto \exp(z)$, $z\mapsto \log(z)$, $\max\{z,c\}$ and $\min\{z,c\}$ for some constant $c\in\mathbb{R}$. However, combined with the simple addition and scaling rules in Proposition~\ref{prop:addition}, this rule can be applied to a much richer set of operations and thereby enable the propagation of superposition relaxations through a broad range of factorable functions:
\begin{itemize}
\item In light of the previous discussions in Sections~\ref{sec:cvxridge} \& \ref{sec:ncvxridge}, composition with any nonmonotonic or nonconvex univariate functions $\varphi$ is possible, as long as $\varphi$ is decomposable as a sum of nondecreasing/nonincreasing convex/concave terms as in \eqref{eq:monotonic_decomp} and \eqref{eq:ncvxdecomp}. This enables both even and odd power terms $z\mapsto z^p$, hyperbolic functions such as $z\mapsto \tanh(z)$, and trigonometric functions such as $z\mapsto \cos(z)$ and $z\mapsto\sin(z)$.

\item A similar approach can be applied to enable binary product operations, $f({\bm x}) = g({\bm x})\times h({\bm x})$, either via DC-decomposition or log-transform:
\begin{align*}
g(\bm x)\times h(\bm x) =&\ \left(\dfrac{g(\bm x)+h(\bm x)}{2}\right)^2 - \left(\dfrac{g(\bm x)-h(\bm x)}{2}\right)^2\\
g(\bm x)\times h(\bm x) =&\ \exp\left(\log\left(g(\bm x)\right)+\log\left(h(\bm x)\right)\right).
\end{align*}
To compute tighter product relaxations or extend their range of application (in the case of the log-transform), the factors $g$ and $h$ may also be rescaled according to their respective ranges.

\item Binary min and max operations too can be readily decomposed. Recalling the identities
\begin{align*}
\max\{g(\bm x),h(\bm x)\} &\ = g(\bm x) + \max\{(h(\bm x)-g(\bm x)),0\}\\ &\ = h(\bm x) + \max\{(g(\bm x)-h(\bm x)),0\},
\end{align*}
a decomposition that preserves symmetry may be obtained by combining scaling, addition and composition with the nondecreasing convex univariate $z\mapsto\max\{z,0\}$ as follows
\begin{align*}
\max\{g(\bm x),h(\bm x)\} =\ & \frac{1}{2}\left[ (g(\bm x)+h(\bm x)) + \max\{(g(\bm x)-h(\bm x)),0\} \right.\\ & \quad \left. + \max\{(h(\bm x)-g(\bm x)),0\} \right].
\end{align*}
Similarly for $\min\{g(\bm x),h(\bm x)\}$,
\begin{align*}
\min\{g(\bm x),h(\bm x)\} =\ & \frac{1}{2}\left[ (g(\bm x)+h(\bm x)) + \min\{(g(\bm x)-h(\bm x)),0\} \right.\\ & \quad \left. + \min\{(h(\bm x)-g(\bm x)),0\} \right].
\end{align*}
\end{itemize}

\subsection{Range-tightening intersection}
\label{sec:tighten}

We discussed in Sections~\ref{sec:cvxridge}--\ref{sec:composition} that the composition operation between a superposition relaxation with a univariate outer function that is both monotonic and convex or concave can be constructed in such a way that the range of the resulting superposition relaxations matches that of the composite function. For nonmonotonic or nonconvex univariate outer functions, on the other hand, this composition entails a decomposition into a sum of monotonic, convex or concave, terms, which generally leads to range overestimation upon composition. This behavior was illustrated in Example~\ref{exm:tanh} for the function $f({\bm x})=\tanh(x_1+x_2)$. 

However, the exact range of many nonmonotonic/nonconvex univariate functions of interest on any given domain is known a priori. The following corollary of Theorem~\ref{thm:cvxcomposition} provides a mechanism for intersecting a superposition relaxation with such prior bounds. 

\begin{corollary}
\label{cor:intersection}
Let $(f^{\rm u},f^{\rm o})_{\bm X}$ be a superposition relaxation of the function $f:{\bm X}\to\mathbb{R}$ on ${\bm X}\in\mathbb{IR}^n$, and suppose that a bound $F\in\mathbb{IR}$ is also available that encloses the range of $f$ on ${\bm X}$. Then, a superposition relaxation $(f^{\rm u,ref},f^{\rm o,ref})_{\bm X}$ of $f$ on ${\rm X}$ is obtained with the following nonzero summands, for $i=1\ldots n$:
\begin{align*}
f^{\rm u,ref}_i(x_i) =\ & \max\left\{f^{\rm u}_i(x_i)-\underline{f}^{\rm u}_i(X_i)+\underline{f}^{\rm u}({\bm X}),\underline{F}\right\} - (1-\theta_i^{\rm u})\,\max\left\{\underline{f}^{\rm u}({\bm X}),\underline{F}\right\}, \quad i\in\mathcal{P}^{\rm u}\\
f^{\rm o,ref}_i(x_i) =\ & \min\left\{f^{\rm o}_i(x_i)-\overline{f}^{\rm o}_i(X_i)+\overline{f}^{\rm o}({\bm X}),\overline{F}\right\} - (1-\theta_i^{\rm o})\,\min\left\{\overline{f}^{\rm o}({\bm X}),\overline{F}\right\}, \quad i\in\mathcal{P}^{\rm o}.
\end{align*}
\end{corollary}

\begin{proof}
Noting that the univariate function $z\mapsto\max\{z,\underline{F}\}$ is nondecreasing and convex on the range $[\underline{f}^{\rm u}({\bm X}),\overline{f}^{\rm o}({\bm X})]$ of the superposition relaxation $(f^{\rm u},f^{\rm o})_{\bm X}$, the validity of the summands $f^{\rm u,ref}_i$ is a direct consequence of Theorem~\ref{thm:cvxcomposition}. The validity of the summands $f^{\rm o,ref}_i$ follows by symmetry, noting that $z\mapsto\min\{z,\overline{F}\}$ is nonincreasing and concave on $[\underline{f}^{\rm u}({\bm X}),\overline{f}^{\rm o}({\bm X})]$.
\end{proof}

\begin{example}
\label{exm:tanh2}
We revisit the superposition relaxation constructed in Example~\ref{exm:tanh} for the ridge function $f({\bm x})=\tanh(x_1+x_2)$ on the domain ${\bm X}=[-0.5,1]^2$ by intersecting it with the a priori bound $F=[\tanh(-1),\tanh(2)]$ based on Corollary~\ref{cor:intersection}. The right plot of Figure~\ref{fig:example2} illustrates that the maximum of the superposition overestimator matches the a priori upper bound, unlike the original superposition overestimator in the left plot. However, this is also seen to be detrimental to the sharpness of the superposition overestimator towards the lower corner $x_1=x_2=-0.5$, where the original superposition overestimator used to connect with the ridge function. On the other hand, the two superposition underestimators are identical, which is expected since the minimum of the ridge function at $x_1=x_2=-0.5$ is correctly tracked.
\end{example}

With the range-tightening intersection given in Corollary~\ref{cor:intersection}, we can construct superposition relaxations whose ranges are always at least as tight as interval bounds derived from natural interval extensions. However, because this construction relies on outer-composition with either a $\min$ or $\max$ function which is not everywhere continuously differentiable, it will impede the quadratic convergence of the superposition relaxations, in general. We formalize these convergence considerations in the next section.

\section{Convergence Analysis}
\label{sec:convergence}

We are interested in studying the convergence of superposition relaxations constructed via the arithmetic rules presented in Section~\ref{sec:arithmetic}. The concept of Hausdorff convergence in traditional interval arithmetic \citep{Moore1979,Alefeld2000}, extends naturally to superposition relaxations, where it bounds the distance between the minima of $f$ and $f^{\rm u}$ and the maxima of $f$ and $f^{\rm o}$ on given variable subsets. 

\begin{definition}
\label{def:hausdorff_convergence}
Let $f:\mathcal{X}\subset\mathbb{R}^n\to\mathbb{R}$ be a continuous function. The superposition relaxations $(f^{\rm u},f^{\rm o})_{{\bm X}\subset\mathcal{X}}$ of $f$ are said to have Hausdorff convergence of order $\beta>0$ if there exists a constant $\tau<\infty$ such that, for all ${\bm X}\subset\mathcal{X}$, ${\bm X}\in\mathbb{IR}^n$,
\begin{align}
\label{eq:hausdorff_distance}
\wid{H_f(\bm X)} - \wid{f(\bm X)} \leq \tau\, {\rm diam}(\bm X)^{\beta},
\end{align}
where $H_f:\mathbb{IR}^n\to\mathbb{IR}$ is the inclusion function such that $H_f(\bm X) \coloneqq \displaystyle \Big[\inf_{\bm x\in\bm X} f^{\rm u}({\bm x}), \sup_{\bm x\in\bm X} f^{\rm o}({\bm x})\Big]$.
\end{definition}

However, that convergence in the Hausdorff sense does not give information about the gap between $f$ and either $f^{\rm u}$ or $f^{\rm o}$ for given points in ${\bm X}$. Following \citep{Bompadre2012,Bompadre2013}, we consider a stronger notion of convergence based on the maximal difference of $f$ with $f^{\rm u}$ and $f^{\rm o}$ on all points of ${\bm X}$.

\begin{definition}
\label{def:pointwise_convergence}
Let $f:\mathcal{X}\subset\mathbb{R}^n\to\mathbb{R}$ be a continuous function. The superposition relaxations $(f^{\rm u},f^{\rm o})_{{\bm X}\subset\mathcal{X}}$ of $f$ are said to have {\em pointwise convergence of order $\alpha>0$} if there exists a constant $\tau<\infty$ such that, for all ${\bm X}\subset\mathcal{X}$, ${\bm X}\in\mathbb{IR}^n$,
\begin{align}
\label{eq:pointwise_distance}
\sup_{\bm x \in \bm X} \left|f(\bm x)-f^{\rm u}(\bm x)\right| \leq \tau\, {\rm diam}(\bm X)^{\alpha} \quad\text{and}\quad \sup_{\bm x \in \bm X} \left|f(\bm x)-f^{\rm o}(\bm x)\right| \leq \tau\, {\rm diam}(\bm X)^{\alpha}.
\end{align}
\end{definition}

\begin{remark}
\label{rmk:hausdorff_vs_pointwise}
Pointwise convergence is indeed stronger than Hausdorff convergence in the sense that the inclusion function associated with a scheme with pointwise convergence $\alpha$ necessarily has Hausdorff convergence of order $\beta\geq \alpha$; see, e.g., Theorem~1 in \citep{Bompadre2012} for a proof.
\end{remark}

It is clear that the initialization operation ${\bm x}\mapsto x_i$ of a superposition relaxation as in \eqref{eq:initialization} has an arbitrarily high pointwise convergence order since it is exact. Moreover, the following proposition summarizes results on the propagation of the pointwise convergence order of superposition relaxations through affine operations. A proof for the addition operation is similar to that of Theorem~3 in \citep{Bompadre2012} for McCormick relaxations, which is independent of the convexity or concavity of the estimators; and a proof for the scaling operation proceeds analogously.

\begin{proposition}
\label{prop:addition_convergence}
Let $g,h:\mathcal{X}\subset\mathbb{R}^n\to\mathbb{R}$ and assume that their respective superposition relaxations $(g^{\rm u},g^{\rm o})_{{\bm X}\subset\mathcal{X}}$ and $(h^{\rm u},h^{\rm o})_{{\bm X}\subset\mathcal{X}}$ have pointwise convergence of order $\alpha_g, \alpha_h>0$, respectively. Then, the superposition relaxations constructed in Proposition~\ref{prop:addition} for the addition operation $g+h$ has pointwise convergence of order $\min\{\alpha_g,\alpha_h\}$ and for the scaling operation $ag+b$ has pointwise convergence of order $\alpha_g$.
\end{proposition}

The main focus in the rest of this section is to determine how the convergence order of superposition relaxations propagates through composition. The following theorem states that quadratic pointwise convergence is preserved through composition with any monotonic, convex or concave function that is continuously differentiable with Lipschitz-continuous derivatives.

\begin{theorem}   
\label{thm:cvxcomposition_convergence}
Let $g:\mathcal{X}\subset\mathbb{R}^n\to\mathbb{R}$, and assume that its superposition relaxations $(g^{\rm u},g^{\rm o})_{{\bm X}\subset\mathcal{X}}$ have pointwise convergence of order $\alpha_g\geq 2$. Let $\varphi:Z\to\mathbb{R}$ be a continuously-differentiable function on $Z\supseteq [\underline{g}^{\rm u}(\mathcal{X}),\overline{g}^{\rm o}(\mathcal{X})]$ with Lipschitz-continuous derivatives, and assume $\varphi$ is both convex and monotonic. Then, the superposition relaxations $(f^{\rm u},f^{\rm o})_{{\bm X}\subset\mathcal{X}}$ for the composite function $f=\varphi\circ g$ constructed in Theorem~\ref{thm:cvxcomposition} have pointwise convergence of order $\alpha_f\geq 2$.
\end{theorem}

The following lemma is used in the proof of Theorem~\ref{thm:cvxcomposition_convergence}.

\begin{lemma}   
\label{lem:estimator_width}
Let $(g^{\rm u},g^{\rm o})_{{\bm X}\subset\mathcal{X}}$ be superposition relaxations for the function $g:\mathcal{X}\subset\mathbb{R}^n\to\mathbb{R}$. Assume $g$ to be Lipschitz continuous on $\mathcal{X}$ with Lipschitz constant $L_g$ such that, for all $\bm X\in \mathcal{X}$, $\bm X\in\mathbb{IR}^n$,
\begin{align*}
\wid{g(\bm X)} \leq L_g\, {\rm diam}(\bm X).
\end{align*}
If $(g^{\rm u},g^{\rm o})_{{\bm X}\subset\mathcal{X}}$ has pointwise convergence of order $\alpha_g\geq 1$, then there exists a constant $L'_g$ such that, for all $i=1,\ldots,n$ and all $\bm X\in \mathcal{X}$, $\bm X\in\mathbb{IR}^n$,
\begin{align*}
\wid{g_i^{\rm u}(X_i)} \leq L'_g\, {\rm diam}(\bm X) \quad\text{and}\quad \wid{g_i^{\rm o}(X_i)} \leq L'_g\, {\rm diam}(\bm X).
\end{align*}
\end{lemma}

\begin{proof}
We consider the case of overestimator summands $g_i^{\rm o}$, noting that a proof for the underestimator summands $g_i^{\rm u}$ proceeds similarly. By separability of the superposition overestimator $g^{\rm o}$, we have for each $i=1,\ldots,n$,
\begin{align*}
\wid{g_i^{\rm o}(X_i)} \leq\ & \wid{g^{\rm o}(\bm X)}\\
\leq\ & \wid{g(\bm X)} + \sup_{\bm x\in\bm X} \left|g^{\rm o}(\bm x) - g(\bm x)\right| - \inf_{\bm x\in\bm X} \left|g^{\rm o}(\bm x) - g(\bm x)\right|\\
\leq\ & \wid{g(\bm X)} + \sup_{\bm x\in\bm X} \left|g^{\rm o}(\bm x) - g(\bm x)\right|.
\end{align*}
Finally, by Lipschitz continuity of $g$ on $\mathcal{X}$ and pointwise convergence of $(g^{\rm u},g^{\rm o})_{{\bm X}\subset\mathcal{X}}$ in $\mathcal{X}$, we have
\begin{align*}
\forall\bm X\in\mathcal{X}, \bm X\in\mathbb{IR}^n, \quad \wid{g_i^{\rm o}(X_i)} \leq\ & L_g\, {\rm diam}(\bm X) + \tau_g \, {\rm diam}(\bm X)^{\alpha_g}.
\end{align*}
for some $\tau_g<\infty$, which proves the result.
\end{proof}

\begin{proof}[Proof of Theorem~\ref{thm:cvxcomposition_convergence}]
We prove that the scheme $(f^{\rm u},f^{\rm o})_{{\bm X}\subset\mathcal{X}}$ has quadratic pointwise convergence in the case of a nondecreasing and convex $\varphi$ and note that the result in the nonincreasing case is proved analogously. We start by decomposing the estimation error as
\begin{align}
\sup_{\bm x \in \bm X} \left|\varphi \circ {g}(\bm x)-\sum_{i=1}^n f^{\rm u}_i(x_i)\right| \leq\ & \sup_{\bm x \in \bm X} \left| \varphi \circ {g}(\bm x) - \varphi\left( \sum_{i=1}^n g^{\rm u}_i(x_i)\right) \right| \nonumber\\ & \qquad\qquad + \sup_{\bm x \in \bm X} \left| \varphi\left( \sum_{i=1}^n g^{\rm u}_i(x_i)\right) - \sum_{i=1}^n f^{\rm u}_i(x_i)\right|\label{eq:supcvxunder}\\
\text{\!\!\!and}\ \ \sup_{\bm x \in \bm X} \left|\varphi \circ {g}(\bm x)-\sum_{i=1}^n f^{\rm o}_i(x_i)\right| \leq\ & \sup_{\bm x \in \bm X}
\left| \varphi \circ {g}(\bm x) - \varphi\left( \sum_{i=1}^n g^{\rm o}_i(x_i)\right) \right| \nonumber\\ & \qquad\qquad + \sup_{\bm x \in \bm X} \left| \varphi\left( \sum_{i=1}^n g^{\rm o}_i(x_i)\right) - \sum_{i=1}^n f^{\rm o}_i(x_i)\right|.\label{eq:supcvxover}
\end{align}
Since $\varphi$ is continuously differentiable on $Z$, let $L_\varphi$ denote its Lipschitz constant so that
\begin{align*}
\sup_{\bm x \in \bm X} \left| \varphi \circ {g}(\bm x) - \varphi\left( \sum_{i=1}^n g^{\rm u}_i(x_i)\right) \right| \leq\ & L_\varphi \sup_{\bm x \in \bm X}
\left| g(\bm x) - \sum_{i=1}^n g^{\rm u}_i(x_i) \right| ,
\end{align*}
and since $(g^{\rm u}_{\bm X},g^{\rm o}_{\bm X})_{{\bm X}\subset\mathcal{X}}$ has pointwise convergence of order $\alpha_g\geq 2$,
\begin{align*}
\sup_{\bm x \in \bm X} \left| \varphi \circ {g}(\bm x) - \varphi\left( \sum_{i=1}^n g^{\rm u}_i(x_i)\right) \right| \leq\ & L_\varphi \tau_g\, {\rm diam}(\bm X)^{\alpha_g}. 
\end{align*}
Similarly, we have
\begin{align*}
\sup_{\bm x \in \bm X} \left| \varphi \circ {g}(\bm x) - \varphi\left( \sum_{i=1}^n g^{\rm o}_i(x_i)\right) \right| \leq\ & L_\varphi \tau_g\, {\rm diam}(\bm X)^{\alpha_g}. 
\end{align*}
Next, consider the second term in the right-hand side of \eqref{eq:supcvxunder}. We introduce the notation $\delta_i^{\rm u}(x_i) \coloneqq g^{\rm u}_i(x_i)-\underline{g}^{\rm u}_i(X_i)$ for $i\in\mathcal{P}^{\rm u}$, noting that $g^{\rm u}_i(x_i)-\underline{g}^{\rm u}_i(X_i)=0$ for $i\notin\mathcal{P}^{\rm u}$. With $\varphi$  convex nondecreasing in Theorem~\ref{thm:cvxcomposition}, we have
\begin{align*}
\varphi\left( \sum_{i=1}^n g^{\rm u}_i(x_i)\right) &- \sum_{i=1}^n f^{\rm u}_i(x_i)\\
=\ & \varphi\left( \sum_{i\in\mathcal{P}^{\rm u}} \delta^{\rm u}_i(x_i) + \underline{g}^{\rm u}({\bm X})\right) - \sum_{i\in\mathcal{P}^{\rm u}} \left(\varphi\left( \delta^{\rm u}_i(x_i) + \underline{g}^{\rm u}({\bm X})\right) - (1-\theta^{\rm u}_i)\underline{g}^{\rm u}({\bm X})\right).
\end{align*}
Second-order Taylor expansion of the composite terms gives
\begin{align*}
\varphi\left( \sum_{i\in\mathcal{P}^{\rm u}} \delta^{\rm u}_i(x_i) + \underline{g}^{\rm u}({\bm X})\right) =\ & \varphi\left( \underline{g}^{\rm u}({\bm X})\right) + \varphi'\left( \underline{g}^{\rm u}({\bm X})\right) \sum_{i\in\mathcal{P}^{\rm u}} \delta^{\rm u}_i(x_i)\\ & + \int_{\underline{g}^{\rm u}({\bm X})}^{\sum_{i\in\mathcal{P}^{\rm u}} \delta^{\rm u}_i(x_i)+\underline{g}^{\rm u}({\bm X})} \varphi''(\zeta)\, \left(\sum_{i\in\mathcal{P}^{\rm u}} \delta^{\rm u}_i(x_i)+\underline{g}^{\rm u}({\bm X})-\zeta\right)\, d\zeta\\
\leq\ & \varphi\left( \underline{g}^{\rm u}({\bm X})\right) + \varphi'\left( \underline{g}^{\rm u}({\bm X})\right) \sum_{i\in\mathcal{P}^{\rm u}} \delta^{\rm u}_i(x_i) + \frac{L_{\varphi'}}{2}\left(\sum_{i\in\mathcal{P}^{\rm u}} \delta^{\rm u}_i(x_i)\right)^2
\end{align*}
and, for all $i\in\mathcal{P}^{\rm u}$,
\begin{align*}
\varphi\left( \delta^{\rm u}_i(x_i) + \underline{g}^{\rm u}({\bm X})\right) =\ & \varphi\left( \underline{g}^{\rm u}({\bm X})\right) + \varphi'\left( \underline{g}^{\rm u}({\bm X})\right) \delta^{\rm u}_i(x_i)\\ & + \int_{\underline{g}^{\rm u}({\bm X})}^{\delta^{\rm u}_i(x_i)+\underline{g}^{\rm u}({\bm X})} \varphi''(\zeta)\, \left(\delta^{\rm u}_i(x_i)+\underline{g}^{\rm u}({\bm X})-\zeta\right)\, d\zeta\\
\geq\ & \varphi\left( \underline{g}^{\rm u}({\bm X})\right) + \varphi'\left( \underline{g}^{\rm u}({\bm X})\right) \delta^{\rm u}_i(x_i) - \frac{L_{\varphi'}}{2} \delta^{\rm u}_i(x_i)^2
\end{align*}
where $L_{\varphi'}$ is the Lipschitz constant of $\varphi'$ on $Z$. Gathering the terms, and noting that $\delta^{\rm u}_i(x_i) \leq \theta^{\rm u}_i\wid{g^{\rm u}(\bm X)} \leq \theta^{\rm u}_i L_g$ for some $L_g<\infty$ by Lemma~\ref{lem:estimator_width}, we obtain
\begin{align*}
& \varphi\left( \sum_{i\in\mathcal{P}^{\rm u}} \delta^{\rm u}_i(x_i) + \underline{g}^{\rm u}({\bm X})\right) - \sum_{i\in\mathcal{P}^{\rm u}} \left(\varphi\left( \delta^{\rm u}_i(x_i) + \underline{g}^{\rm u}({\bm X})\right) - (1-\theta^{\rm u}_i)\underline{g}^{\rm u}({\bm X})\right)\\
& \qquad\qquad\qquad \leq \frac{L_{\varphi'}}{2}\left(\sum_{i\in\mathcal{P}^{\rm u}} \delta^{\rm u}_i(x_i)\right)^2 + \frac{L_{\varphi'}}{2} \sum_{i\in\mathcal{P}^{\rm u}} \delta^{\rm u}_i(x_i)^2
\leq L_{\varphi'} L_g^2\, {\rm diam}(\bm X)^2,
\end{align*}
thereby establishing that the superposition underestimator has quadratic pointwise convergence. Lastly, we apply a similar approach to bound the second term in the right-hand side of \eqref{eq:supcvxover}, introducing the notation $\delta_i^{\rm o}(x_i) \coloneqq g^{\rm o}_i(x_i)-\overline{g}^{\rm o}_i(X_i)$ for $i\in\mathcal{P}^{\rm o}$. For a convex nondecreasing $\varphi$ in Theorem~\ref{thm:cvxcomposition}, we have
\begin{align*}
\varphi\left( \sum_{i=1}^n g^{\rm o}_i(x_i)\right) - \sum_{i=1}^n f^{\rm o}_i(x_i)
=\ & \varphi\left( \sum_{i\in\mathcal{P}^{\rm o}} \delta^{\rm o}_i(x_i) + \overline{g}^{\rm o}({\bm X})\right) - \sum_{i\in\mathcal{P}^{\rm o}} \theta^o_i\varphi\left( \frac{\delta^{\rm o}_i(x_i)}{\theta^{\rm o}_i} + \overline{g}^{\rm o}({\bm X})\right).
\end{align*}
Second-order Taylor expansion of the composite terms gives
\begin{align*}
\varphi\left( \sum_{i\in\mathcal{P}^{\rm o}} \delta^{\rm o}_i(x_i) + \overline{g}^{\rm o}({\bm X})\right) 
\leq\ & \varphi\left( \overline{g}^{\rm o}({\bm X})\right) + \varphi'\left( \overline{g}^{\rm o}({\bm X})\right) \sum_{i\in\mathcal{P}^{\rm o}} \delta^{\rm o}_i(x_i) + \frac{L_{\varphi'}}{2}\left(\sum_{i\in\mathcal{P}^{\rm o}} \delta^{\rm o}_i(x_i)\right)^2
\end{align*}
and, for all $i\in\mathcal{P}^{\rm o}$,
\begin{align*}
\theta^{\rm o}_i\varphi\left( \frac{\delta^{\rm o}_i(x_i)}{\theta^{\rm o}_i} + \overline{g}^{\rm o}({\bm X})\right) \geq \ & \theta^{\rm o}_i \varphi\left( \overline{g}^{\rm o}({\bm X})\right) + \varphi'\left( \overline{g}^{\rm o}({\bm X})\right) \delta^{\rm o}_i(x_i) - \frac{L_{\varphi'}}{2\theta^{\rm o}_i} \delta^{\rm o}_i(x_i)^2.
\end{align*}
Finally, noting that by Lemma~\ref{lem:estimator_width}
\begin{align*}
\exists L_g<\infty:\ \forall x_i\in X_i, \quad \frac{\delta^{\rm o}_i(x_i)}{\theta^{\rm o}_i} \leq \wid{g^{\rm o}(\bm X)} \leq L_g\, {\rm diam}(\bm X)
\end{align*}
and combining the previous inequalities, we get
\begin{align*}
& \varphi\left( \sum_{i\in\mathcal{P}^{\rm o}} \delta^{\rm o}_i(x_i) + \overline{g}^{\rm o}({\bm X})\right) - \sum_{i\in\mathcal{P}^{\rm o}} \theta^o_i\varphi\left( \frac{\delta^{\rm o}_i(x_i)}{\theta^{\rm o}_i} + \overline{g}^{\rm o}({\bm X})\right) \leq L_{\varphi'} L_g^2\, {\rm diam}(\bm X)^2,
\end{align*}
showing that the superposition overestimator also has quadratic pointwise convergence.
\end{proof}

The results in Proposition~\ref{prop:addition_convergence} and Theorem~\ref{thm:cvxcomposition_convergence} confirm that applying the arithmetic rules in Proposition~\ref{prop:addition} and Theorem~\ref{thm:cvxcomposition} leads to quadratically convergent schemes of superposition relaxations under mild assumptions. Indeed, many univariate functions $\varphi$ of interest are continuously differentiable with Lipschitz continuous derivatives, and this also includes the terms $\varphi(\min\{z,\sigma^\ast\})$ and $\varphi(\max\{z,\sigma^\ast\})$ in the decomposition \eqref{eq:nonmonotonic_underest} and the terms $\varphi_k$ in the decomposition \eqref{eq:ncvxdecomp}---see Example~\ref{exm:sqr2} below for an illustration. As expected, however, the relaxations may no longer enjoy quadratic pointwise convergence around points where a univariate term is not differentiable, for instance with the terms $x\mapsto\sqrt{z}$ and $z\mapsto|z|$ around $z=0$---see Example~\ref{exm:abs}. The application of the intersection operation in Corrolary~\ref{cor:intersection} may also impede the quadratic pointwise convergence due to the nondifferentiable $\min$ and $\max$ terms.

\begin{example}
\label{exm:sqr2}
We revisit the superposition relaxation constructed in Example~\ref{exm:sqr} for the ridge function $f({\bm x})=(x_1+x_2)^2$, now to analyze its pointwise convergence order. We start with the initial domain ${\bm X}_1=[-1,2]^2$, which we contract around given points ${\bm x}^\ast\in\bm X$ such that ${\bm X}_\rho \coloneqq \rho{\bm X}_1 + (1-\rho){\bm x}^\ast$ for $\rho\in (0,1]$. We rely on \eqref{eq:monotonic_overest} to construct the superposition overestimators; while we apply the decomposition in \eqref{eq:nonmonotonic_underest} to compute the superposition underestimators since $f$ can be nonmonotonic on ${\bm X}_\rho$. In particular, the latter adds superposition underestimators for the terms $\min\{x_1+x_2,0\}^2$ and $\max\{x_1+x_2,0\}^2$, which are continuously differentiable with Lipsichtz-continuous derivatives over ${\bm X}_1$. We can therefore expect for the superposition relaxations to still enjoy quadratic pointwise convergence per Theorem~\ref{thm:cvxcomposition_convergence}, despite the use of $\min$ and $\max$ operation. This is illustrated on the left plot of Figure~\ref{fig:example4}, where the error shrinks proportionally to $\rho^2$, even around the point $x_1^\ast+x_2^\ast=0$ where the derivative is not differentiable. This illustration also confirms that the quadratic convergence order in Theorem~\ref{thm:cvxcomposition_convergence} is sharp, even for propagation through composition with analytic functions.
\end{example}

\begin{figure}[tb]
\centering
\includegraphics[width=0.98\linewidth]{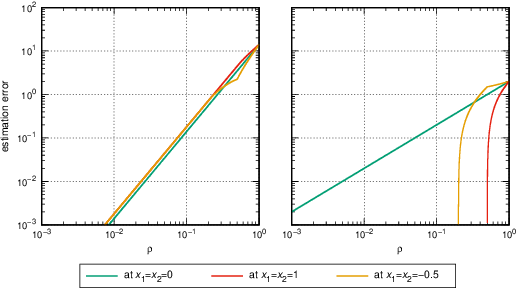}
\caption{Pointwise convergence analysis of a scheme of superposition relaxations for the ridge function $f({\bm x})=(x_1+x_2)^2$ (left plot) and $f({\bm x})=|x_1+x_2|$ (right plot) around three reference points in the domain ${\bm X}=[-1,2]^2$.}
\label{fig:example4}
\end{figure}

\begin{example}
\label{exm:abs}
We consider the ridge function $f({\bm x})=|x_1+x_2|$, which is not differentiable at points with $x_1+x_2=0$ and therefore Theorem~\ref{thm:cvxcomposition_convergence} may not apply. We again start with the initial domain ${\bm X}_1=[-1,2]^2$ and contract it around given points ${\bm x}^\ast\in\bm X$ as ${\bm X}_\rho \coloneqq \rho{\bm X}_1 + (1-\rho){\bm x}^\ast$ for $\rho\in (0,1]$. We construct the superposition underestimators and overestimators using \eqref{eq:monotonic_underest} and \eqref{eq:monotonic_overest}, respectively. The right plot of Figure~\ref{fig:example4} confirms the loss of quadratic pointwise convergence around points of nondifferentiability such as $x_1=x_2=0$, where the error shrinks proportionally to $\rho$ instead. In contrast, the function is locally linear around differentiable points, such as $x_1=x_2=1$ or $x_1=x_2=-0.5$, and can thus be estimated exactly by the superposition relaxations for small enough $\rho$ values.
\end{example}

\section{Parameterization of Superposition Relaxations}
\label{sec:parameterization}

An effective implementation of superposition relaxation hinges on the ability to propagate the univariate summands through the arithmetic rules developed in Section~\ref{sec:arithmetic} and in such a way that the asymptotic convergence properties established in Section~\ref{sec:convergence} can be preserved. In this section, we consider finite parameterizations of the univariate summands. Three desirable properties for a practical parameterization include:
\begin{enumerate}
\item[P1.] Ability to recover the extrema of a parameterized univariate function efficiently;
\item[P2.] Invariance of the parameterization under affine operations;
\item[P3.] Uniform approximation of any univariate continuous function from below and above by a sequence of parameterized univariate functions with increasing number of parameters. 
\end{enumerate}
In addition to computing the exact range of a superposition relaxation, the first property P1 is also relevant for the composition rule established in Theorem~\ref{thm:cvxcomposition} as it uses either the minimum or maximum of the univariate summands to propagate the superposition relaxations. The second property P2 ensures that superposition relaxations can be propagated exactly through addition and scaling operations as defined in Proposition~\ref{prop:addition} and that their pointwise convergence order propagates per the result of Proposition~\ref{prop:addition_convergence}. As for the third property P3, it guarantees that the approximation error introduced by the parameterization can be controlled to an arbitrary precision upon increasing the number of parameters, e.g. via subpartitioning. Additionally, to preserve the quadratic convergence order of the superposition relaxations through the composition rule per Theorem~\ref{thm:cvxcomposition_convergence}, the approximation error of the parameterized univariate functions itself should be quadratically convergent. 

\begin{remark}
It is work noting that one could in principle propagate superposition estimators in a pointwise manner, which would remove the need to parameterize the univariate summands. However, doing so would require the propagation of supporting bounds on the range of each operation in order to enable the composition; see the composition rule in Theorem~\ref{thm:cvxcomposition} where bounds $\underline{g}_i^{\rm u}$ and $\overline{g}_i^{\rm o}$ are needed for each summand. This reliance on supporting bounds is completely symmetrical with McCormick relaxations \cite{McCormick1976}, which mandate interval bounds to be available alongside the pointwise convex underestimator and concave overestimator to compute product and composition operations. Nevertheless, relying on such interval bounds could weaken the superposition relaxations, albeit still preserving their quadratic convergence \cite{Bompadre2012}. Since the relaxations here are a superposition of univariate functions by construction, it is therefore both beneficial and tractable to parameterize these univariate summands. The situation with McCormick relaxations is different because the convex underestimator and concave overestimator counterparts are non-separable functions in general.
\end{remark}

In the remainder of this section, we discuss two finite parameterizations that consist of piecewise-constant univariates (Section~\ref{sec:pwc}) and continuous piecewise-linear univariates (Section~\ref{sec:pwl}). Alternatively, orthogonal polynomials could also be used, such as Chebyshev series up to a given maximal order as in Chebyshev models \cite{Rajyaguru2017} or even approximating the univariate summands to machine precision \citep{Trefethen2007,Driscoll2014}. 

\subsection{Piecewise-Constant Parameterization}
\label{sec:pwc}

Perhaps the simplest parameterization for the univariate summands entails piecewise-continuous functions on a set partition of the variable range. In particular, superposition relaxations with piecewise-constant summands on an equidistant partition matches the interval superposition models first investigated in \citep{Zha2018,Su2019}.

Formally, any piecewise-constant univariate function $u:X\in\mathbb{IR}\to\mathbb{R}$ over the partition $X_1,\ldots,X_N\in\mathbb{IR}$, with $N>0$, can be encoded by the vector of parameters ${\bm \nu}\coloneqq [\nu_1 \cdots \nu_N]^\intercal\in\mathbb{R}^N$ such that
\begin{align}
\label{eq:pwc}
\forall x\in X_k,\ k=1\ldots N, \quad u(x) = \nu_k.
\end{align}
The complexity of recovering the minimum or maximum of the piecewise-constant function \eqref{eq:pwc} is proportional to the partition size $N$ as it entails a linear traversal of $\bm \nu$ (property P1). Additionally, affine and composition operations on piecewise-constant under- and overestimators are readily automated using interval arithmetic rules \citep{Moore1979} on each subdomain $X_k$, thereby guaranteeing the affine invariance (P2) and uniform approximation (P3) properties. One caveat, however, is that the variables themselves may only be represented with accuracy $\mathcal{O}(\frac{\wid{X}}{N})$ using a piecewise-constant parameterization over an equidistant partition. Therefore, the pointwise convergence of the resulting piecewise-constant superposition relaxations may only be linear in general, despite Theorem~\ref{thm:cvxcomposition_convergence} remaining applicable with piecewise-constant summands for the propagation of quadratically-convergent superposition relaxations through composition.

\subsection{Continuous Piecewise-Linear Parameterization}
\label{sec:pwl}

To retain the quadratic convergence order of a superposition relaxation for sufficiently smooth factorable functions, we next consider parameterizing the univariate summands as continuous piecewise-linear functions and allow for adaptive partitions. Of the possible encodings we choose the segment-based representation, which for a continuous piecewise-linear univariate function $u:X\in\mathbb{IR}\to\mathbb{R}$ over the partition $X_1,\ldots,X_N\in\mathbb{IR}$, with $N>0$, consists of the triplet $(\nu_0,{\bm \delta}, {\bm \sigma})\in\mathbb{R}\times \mathbb{R}^N\times\mathbb{R}^N$ such that
\begin{align}
\label{eq:pwl}
\forall x\in X_k,\ k=1\ldots N, \quad u(x) = \nu_0 + \sum_{j=1}^{k-1} (\sigma_j-\sigma_k)\delta_j +\sigma_k (x-\underline{X}).
\end{align} 
The parameter $\nu_0$ encodes the value of $u$ at $\underline{X}$; $\delta_k$, the width of the subdomain $X_k$; and $\sigma_k$, the slope of $u$ on $X_k$. We prefer this encoding to more conventional ones such as vertex-based here, as it is convenient for the implementation of addition, truncation and composition operations over continuous piecewise-linear functions; see Algorithms~\ref{alg:pwl_addition}--\ref{alg:pwl_composition} below.

Recovering the minimum or maximum of the continuous piecewise-linear function $u$ in \eqref{eq:pwl} entails a simple traversal across the $N+1$ vertices,
\begin{align*}
\underline{u}(X) =& \min \left\{ \nu_0 + \sum_{j=1}^{k-1} \sigma_j\delta_j ~\middle|~ k=1,\ldots, N+1 \right\},\\
\overline{u}(X) =& \max \left\{ \nu_0 + \sum_{j=1}^{k-1} \sigma_j\delta_j ~\middle|~ k=1,\ldots, N+1 \right\}.
\end{align*}
The corresponding computational effort increases linearly with the partition size $N$ (property P1), as was earlier the case with piecewise-constant parameterizations (Section~\ref{sec:pwc}).

\begin{algorithm}[tb]
\caption{Binary sum of continuous piecewise-linear functions using segment-based encoding \label{alg:pwl_addition}}
\small
\smallskip
\begin{center}
\begin{minipage}{1\textwidth}
\textbf{Input:} domain $X\in\mathbb{IR}$, triplets $(\nu_0,{\bm \delta}, {\bm \sigma})\in\mathbb{R}^{2N+1}$ and $(\nu'_0,{\bm \delta}', {\bm \sigma}')\in\mathbb{R}^{2N'+1}$ encoding the continuous piecewise-linear univariates $u$ and $u'$ on $X$\\[0.25em]
\textbf{Main Steps:}
\begin{enumerate}
\setlength{\itemsep}{.15em}\small
\item Set $i=1$, $i'=1$, $\tilde{N}=1$, $\tilde{\nu}_0 = \nu_0 + \nu'_0$, $\tilde{\delta}_1 = \min\{\delta_1,\delta'_1\}$, $\tilde{\sigma}_1 = \sigma_1 + \sigma'_1$
\item If $\delta_i < \delta'_{i'}$:\hspace{2.00em} update $\delta'_{i'} \leftarrow \delta'_{i'} - \delta_i$, $i\leftarrow i+1$ \\
      Else if $\delta_i > \delta'_{i'}$: update $\delta_i \leftarrow \delta_i - \delta'_{i'}$, $i'\leftarrow i'+1$ \\
      Else:\hspace{4.55em} update $i\leftarrow i+1$, $i'\leftarrow i'+1$
\item If $i \leq N$ and $i' \leq N'$: update $\tilde{N}\leftarrow \tilde{N}+1$, set $\tilde{\delta}_{\tilde{N}} = \min\{\delta_i,\delta'_{i'}\}$, $\tilde{\sigma}_{\tilde{N}} = \sigma_i + \sigma'_{i'}$\\ $\hphantom{\text{If $i \leq N$ and $i' \leq N'$:}}$ return to Step 2
\end{enumerate}
\textbf{Output:} triplet $(\tilde{\nu}_0,\tilde{\bm \delta}, \tilde{\bm \sigma})\in\mathbb{R}^{2\tilde{N}+1}$ encoding the continuous piecewise-linear univariate $u+u'$ on $X$
\end{minipage}
\end{center}
\end{algorithm}

\begin{algorithm}[tb]
\caption{Truncation of a continuous piecewise-linear function using segment-based encoding \label{alg:pwl_truncation}}
\small
\smallskip
\begin{center}
\begin{minipage}{1\textwidth}
\textbf{Input:} domain $X\in\mathbb{IR}$, triplet $(\nu_0,{\bm \delta}, {\bm \sigma})\in\mathbb{R}^{2N+1}$ encoding the continuous piecewise-linear univariate $u$ on $X$, constant $c\in\mathbb{R}$\\[0.25em]
\textbf{Main Steps:}
\begin{enumerate}
\setlength{\itemsep}{.15em}\small
\item Set $i=1$, $\tilde{N}=1$, $\tilde{\nu}_0=\max({\nu}_0,c)$, $\xi_{\rm L}=\nu_0$, $\xi_{\rm R}=\nu_0+\sigma_1\delta_1$
\item If $\xi_{\rm L} \leq c$ and $\xi_{\rm R} \leq c$:\hspace{0em} update $\tilde{\delta}_{\tilde{N}}=\delta_i$, $\tilde{\sigma}_{\tilde{N}}=0$ \\
  Else if $\xi_{\rm L} < c$:\hspace{0em} update $\tilde{\sigma}_{\tilde{N}}=0$, $\tilde{N} \leftarrow \tilde{N}+1$, $\tilde{\delta}_{\tilde{N}} = \displaystyle\frac{ \xi_{\rm R} - c }{\sigma_{i}}$, $\tilde{\delta}_{\tilde{N}-1} \leftarrow \delta_{i} - \tilde{\delta}_{\tilde{N}}$, $\tilde{\sigma}_{\tilde{N}}=\sigma_{i}$\\
  Else if $\xi_{\rm R} < c$:\hspace{0em} update $\tilde{\sigma}_{\tilde{N}}=\sigma_i$, $\tilde{N} \leftarrow \tilde{N}+1$,  $\tilde{\delta}_{\tilde{N}} = \displaystyle\frac{ \xi_{\rm R} - c }{\sigma_{i}}$, $\tilde{\delta}_{\tilde{N}-1} \leftarrow \delta_{i} - \tilde{\delta}_{\tilde{N}}$, $\tilde{\sigma}_{\tilde{N}}=0$\\
  Else:\hspace{0em} update: $\tilde{\delta}_{\tilde{N}}=\delta_i$, $\tilde{\sigma}_{\tilde{N}}=\sigma_i$ 
\item If $i < N$: update $i \leftarrow i+1$, $\tilde{N} \leftarrow \tilde{N}+1$,  $\xi_{\rm L} \leftarrow \xi_{\rm R}$, $\xi_{\rm R} \leftarrow \xi_{\rm L} + \delta_i\sigma_i$\\
$\hphantom{\text{If $i < N$:}}$ return to Step 2
\end{enumerate}
\textbf{Output:} triplets $(\tilde{\nu}_0,\tilde{\bm \delta}, \tilde{\bm \sigma})\in\mathbb{R}^{2\tilde{N}+1}$ encoding the continuous piecewise-linear univariate $\max(u,c)$ on $X$
\end{minipage}
\end{center}
\end{algorithm}

The affine invariance property (P2) is also satisfied for continuous piecewise-linear parameterizations. Given the continuous piecewise-linear function $u$ parameterized by $(\nu_0,{\bm \delta}, {\bm \sigma})$, it is easy to see that the affine transformation $au+b$ yields another continuous piecewise-linear function parameterized by $(a\nu_0+b,\bm\delta, a\bm \sigma)$. The sum of two piecewise-linear functions $u+u'$ again yields a continuous piecewise-linear function, obtained by adding the slopes of $u$ and $u'$ over the combination of their partitions. This process is formalized in Algorithm~\ref{alg:pwl_addition}, where the complexity of binary sum increases linearly with the size $N+N'$ of the respective partitions of $u$ and $u'$. Beside affine operations, continuous piecewise-linear functions are also invariant under truncation operation, e.g. $\max(u,c)$ and $\min(u,c)$. An implementation of $\max$ truncation using segment-based encoding is provided in Algorithm~\ref{alg:pwl_truncation}, with a complexity no greater than double the size $N$ of the partition of $u$.

\begin{algorithm}[tb]
\caption{Unary composition of a continuous piecewise-linear function with a convex outer function using segment-based encoding \label{alg:pwl_composition}}
\small
\smallskip
\begin{center}
\begin{minipage}{1\textwidth}
\textbf{Input:} domain $X\in\mathbb{IR}$, triplet $(\nu_0,{\bm \delta}, {\bm \sigma})\in\mathbb{R}^{2N+1}$ encoding the continuous piecewise-linear univariate $u$ on $X$, convex unary function $\varphi$ on $[\underline{u}(X),\overline{u}(X)]$\\[0.25em]
\textbf{Main Steps:}
\begin{enumerate}
\setlength{\itemsep}{.15em}\small
\item Set $i=1$, $\hat{N}=1$, $\check{N}=2$, $\hat{\nu}_0=\check{\nu}_0=\varphi(\nu_0)$, $\xi_{\rm L}=\nu_0$, $\xi_{\rm R}=\nu_0+\sigma_1\delta_1$
\item Compute $\displaystyle \hat{\sigma}_{\hat{N}} = \frac{\varphi(\xi_{\rm R}) - \varphi(\xi_{\rm L})}{\delta_{i}}$, $\check{\sigma}_{\check{N}-1} = \sigma_i \varphi'(\xi_{\rm L})$, $\check{\sigma}_{\check{N}} = \sigma_i \varphi'(\xi_{\rm R})$,\\ $\hphantom{\text{Compute}\ }\hat{\delta}_{\hat{N}} = \delta_{i}$, $\displaystyle \check{\delta}_{\check{N}-1} = \frac{\check{\sigma}_{\check{N}}\delta_i - \varphi(\xi_{\rm R}) + \varphi(\xi_{\rm L})}{\check{\sigma}_{\check{N}} - \check{\sigma}_{\check{N}-1}}$, $\check{\delta}_{\check{N}} = \delta_i - \check{\delta}_{\check{N}-1}$
\item If $\hat{N} < N$: update $i \leftarrow i+1$, $\hat{N}\leftarrow \hat{N}+1$, $\check{N} \leftarrow \check{N}+2$, $\xi_{\rm L} \leftarrow \xi_{\rm R}$, $\xi_{\rm R}\leftarrow \xi_{\rm L} + \sigma_i\delta_i$ \\ $\hphantom{\text{If $\hat{N} < N$:}}$ return to Step 2 
\end{enumerate}
\textbf{Output:} triplets $(\check{\nu}_0,\check{\bm \delta}, \check{\bm \sigma})\in\mathbb{R}^{2\check{N}+1}$ and $(\hat{\nu}_0,\hat{\bm \delta}, \hat{\bm \sigma})\in\mathbb{R}^{2\hat{N}+1}$ encoding a continuous piecewise-linear univariate under- and overestimator of $\varphi\circ u$ on $X$
\end{minipage}
\end{center}
\end{algorithm}

In contrast with an affine or truncation operation, the composition of a continuous piecewise-linear function with a nonlinear outer function is not a continuous piecewise-linear function in general. Instead, the idea is to construct pairs of continuous piecewise-linear functions that bracket such nonlinear composite functions. Algorithm~\ref{alg:pwl_composition} describes one such construction for the composition of a continuous piecewise-linear function $u$ with a convex outer function $\varphi$, where the tangent and secant lines of $\varphi\circ u$ over the partition of $u$ are used as under- and overestimators, respectively. The resulting piecewise-linear overestimator of $\varphi\circ u$ parameterized by $(\hat{\nu}_0,\hat{\bm \delta}, \hat{\bm \sigma})\in\mathbb{R}^{2\hat{N}+1}$ shares the same breakpoints as $u$, whereas the piecewise-linear underestimator parameterized by $(\check{\nu}_0,\check{\bm \delta}, \check{\bm \sigma})\in\mathbb{R}^{2\check{N}+1}$ introduces extra breakpoints at the intersection between adjacent tangent lines, thereby doubling the number of segments compared to $u$. In a variant of Algorithm~\ref{alg:pwl_composition}, one could introduce extra breakpoints in the overestimator of $\varphi\circ u$ as well, or even introduce breakpoints adaptively to keep the approximation error within preset limits like the sandwich algorithm used in polyhedral outer-approximation schemes \citep{Rote1992,Tawarmalani2005}. For composition with a concave outer function $\varphi$, Algorithm~\ref{alg:pwl_composition} can be applied similarly, but assigning the parameterizations $(\hat{\nu}_0,\hat{\bm \delta}, \hat{\bm \sigma})$ and $(\check{\nu}_0,\check{\bm \delta}, \check{\bm \sigma})$ to, respectively, the underestimator and overestimator of $\varphi\circ u$ instead.

\begin{remark}
\label{rmk:never_coarsen}
The resultant triplets computed by any of Algorithms~\ref{alg:pwl_addition}--\ref{alg:pwl_composition} specify a partition of $X$ that is either identical to or a refinement of the partition specified by the input triplet. Moreover, an affine transformation $ax+b$ does not change the partition of $X$. Therefore, for the unary and binary operations considered herein, the propagation of continuous piecewise-linear functions using segment-based encoding always retains or subdivides existing segments, but never coarsens them. 
\end{remark}

The following proposition establishes that Algorithm~\ref{alg:pwl_composition} propagates the quadratic pointwise convergence order of the univariate under- and overestimators under the same regularity conditions for $\varphi$ as in Theorem~\ref{thm:cvxcomposition_convergence}. A complication here is that the number $N$ of segments may vary through the recursive application of composition operations. Instead, we prove the result in terms of the initial number $N_0$ of segments used to partition a variable's domain. Since $N$ can only increase through recursive composition---see Remark~\ref{rmk:never_coarsen}, the relevant segment-width and estimation-error bounds remain controlled with respect to $N_0$.

\begin{proposition}
\label{prop:cvxcomposition_convergence}
Let $u:\mathcal{X}\subset\mathbb{R} \to \mathbb R$ and assume that $\left(u^{\rm u},u^{\rm o}\right)_{X\subset\mathcal{X}}$ is a continuous piecewise-linear estimator, for which each of the $N^{\rm u}\geq N_0$ segments of $u^{\rm u}$ and each of the $N^{\rm o}\geq N_0$ segments of $u^{\rm o}$ has width at most $C_x\frac{\wid{X}}{N_0}$, with constant $C_x>0$ independent of $N_0$ and $X$. Let $\varphi:\mathcal{Z}\to\mathbb{R}$ be continuously differentiable on $\mathcal{Z}\supseteq [\underline{u}^{\rm u}(\mathcal{X}),\overline{u}^{\rm o}(\mathcal{X})]$ with Lipschitz-continuous derivatives, and assume $\varphi$ is convex and non-decreasing. Construct both an underestimator $f^{\rm u}$ of $\varphi\circ u^{\rm u}$ and an overestimator $f^{\rm o}$ of $\varphi\circ u^{\rm o}$ using Algorithm~\ref{alg:pwl_composition}. 
If the estimation error of $\left(u^{\rm u},u^{\rm o}\right)_{X\subset\mathcal{X}}$ satisfies
\[
\sup_{x\in X}\left(u(x)-u^{\mathrm{u}}(x)\right) \leq C_u\left(\frac{\wid{X}}{N_0}\right)^2
\quad \text{and} \quad
\sup_{x\in X}\left(u^{\mathrm{o}}(x)-u(x)\right) \leq C_u\left(\frac{\wid{X}}{N_0}\right)^2
\]
where the constant $C_u$ is independent of $X$ and $N_0$, then the estimation error of $f\coloneqq \varphi\circ u$ by the continuous piecewise-linear estimator $(f^{\rm u},f^{\rm o})_{X\subset\mathcal{X}}$ satisfies
\[
\sup_{x\in X}\bigl(f(x)-f^{\mathrm{u}}(x)\bigr) \leq C_f\left(\frac{\wid{X}}{N_0}\right)^2 
\quad \text{and} \quad
\sup_{x\in X}\bigl(f^{\mathrm{o}}(x)-f(x)\bigr) \leq C_f\left(\frac{\wid{X}}{N_0}\right)^2 
\]
where the constant $C_f$ is again independent of $X$ and $N_0$.
\end{proposition}

Notice that the counterpart of Proposition~\ref{prop:cvxcomposition_convergence} for a convex and non-increasing outer function $\varphi$ follows by symmetry, via the application of Algorithm~\ref{alg:pwl_composition} to construct both an underestimator $f^{\rm u}$ of $\varphi\circ u^{\rm o}$ and an overestimator $f^{\rm o}$ of $\varphi\circ u^{\rm u}$.

The following lemma is used in the proof of Proposition~\ref{prop:cvxcomposition_convergence}. It shows that, in the construction of Algorithm~\ref{alg:pwl_composition}, the approximation error on each partition element is guaranteed to shrink quadratically with the width of that element.

\begin{lemma}   
\label{lem:unicvxtangentsecant}
 Let $\varphi:Z\in\mathbb{IR}\to\mathbb{R}$ be convex or concave and continuously-differentiable on $Z$ with Lipschitz-continuous derivatives. For the secant line $s_\varphi:Z\to\mathbb{R}$ connecting both endpoints $\varphi(\underline{Z})$ and $\varphi(\overline{Z})$ and the tangent line $t_{\varphi,z^\ast}:Z\to\mathbb{R}$ to $\varphi$ at any given point $z^\ast\in Z$, we have
\begin{align}
\label{eq:unitangent}
\sup_{z \in Z} \left|\varphi(z)-t_{\varphi,z^\ast}(z)\right| &\leq \frac{L_{\varphi'}}{2} \wid{Z}^{2}\\
\label{eq:unisecant}
\text{and}\quad \sup_{z \in Z} \left|\varphi(z)-s_\varphi(z)\right| &\leq \frac{3L_{\varphi'}}{2} \wid{Z}^{2} \, , 
\end{align}
where $L_{\varphi'}$ is the Lipschitz constant of $\varphi'$ on $Z$.
\end{lemma}

\begin{proof}
From the regularity assumptions, second-order Taylor expansion of $\varphi$ at any point $z^\ast\in Z$ is such that
\begin{align*}
\varphi\left(z\right) =\ & \varphi\left(z^\ast\right) + \varphi'\left(z^\ast\right) (z - z^\ast) + \int_{z^\ast}^{z} \varphi''(\zeta)\, \left(z-\zeta\right)\, d\zeta,
\end{align*}
from which we have
\begin{align}
\label{eq:uniineq}
\forall z\in Z, \quad \left|\varphi\left(z\right) - \varphi\left(z^\ast\right) - \varphi'\left(z^\ast\right) (z - z^\ast)\right| \leq \frac{L_{\varphi'}}{2}\, \wid{Z}^2 .
\end{align}
It is immediate, therefore, that the tangent line $t_{\varphi,z^\ast} \coloneqq \varphi\left(z^\ast\right) + \varphi'\left(z^\ast\right) (z - z^\ast)$ satisfies \eqref{eq:unitangent}. As for the secant line, we have
\begin{align*}
s_\varphi(z) \coloneqq\ & \varphi(\underline{Z}) + \frac{\varphi(\overline{Z})-\varphi(\underline{Z})}{\wid{Z}} \left(z-\underline{Z}\right)
= \varphi(\underline{Z}) + \varphi'(\zeta) \left(z-\underline{Z}\right)
\end{align*}
for some $\zeta\in Z$, which is a direct consequence of the mean-value theorem. For a convex or concave $\varphi$ on $Z$ in particular, the distance between $\varphi$ and its secant $s_\varphi$ is maximal at $\zeta$ \citep{Rote1992}, and it then follows from \eqref{eq:uniineq} and Lipschitz continuity of $\varphi'$ that
\begin{align*}
\forall z\in Z, \quad \left|\varphi\left(z\right) - s_\varphi\left(z\right)\right| =\ & \left|\varphi\left(z\right) - \varphi(\underline{Z}) - \varphi'(\zeta) \left(z-\underline{Z}\right)\right|\\
\leq\ & \left|\varphi\left(\zeta\right) - \varphi(\underline{Z}) - \varphi'(\underline{Z}) \left(\zeta-\underline{Z}\right)\right| + \left|\varphi'\left(\underline{Z}\right) - \varphi'(\zeta)\right|\: \left|\zeta-\underline{Z}\right|\\
\leq\ & \frac{3L_{\varphi'}}{2} \wid{Z}^2,
\end{align*}
thereby proving \eqref{eq:unisecant}.
\end{proof}

\begin{proof}[Proof of Proposition~\ref{prop:cvxcomposition_convergence}]
We proceed by decomposing the underestimation error bound as 
\begin{align}
\sup_{x \in X} \left|\varphi \circ u(x)-f^{\rm u}(x)\right| \leq & \sup_{x \in X} \left| \varphi \circ u(x) - \varphi\circ u^{\rm u}(x) \right| + \sup_{x \in X} \left| \varphi\circ u^{\rm u}(x) - f^{\rm u}(x)\right|\label{eq:unicvxunder}
\end{align}
and similarly for the overestimation error bound.
Since $\varphi$ is continuously differentiable on $Z$ and denoting its Lipschitz constant by $L_\varphi$, it follows by assumption on the estimation error of $\left(u^{\rm u},u^{\rm o}\right)_{X\subset\mathcal{X}}$ that
\begin{align}
\sup_{x \in X} \left| \varphi \circ u(x) - \varphi\circ u^{\rm u}(x) \right| \leq\ & L_\varphi \sup_{x \in X} \left| u(x) - u^{\rm u}(x) \right| \leq  L_\varphi C_u\left(\frac{\wid{X}}{N_0}\right)^2. \label{eq:unicvxunder2}
\end{align}
Next, we recall that Algorithm~\ref{alg:pwl_composition} proceeds segment-by-segment to construct the continuous piecewise-linear under- and overestimators $f^{\rm u}$, $f^{\rm o}$. 
For $k \in\{1,\ldots,N^{\rm u}\}$, the width of the $k$th segment $X_k$ supporting $u^{\mathrm u}$ is so that $\wid{X_k} \leq C_x\frac{\wid{X}}{N_0}$ by assumption.
By Lemma~\ref{lem:unicvxtangentsecant}, the estimation error of $\varphi\circ u^{\rm u}$ on $X_k$ is quadratic in its width $\wid{X_k}$ when $\varphi$ is continuously differentiable with Lipschitz-continuous derivatives. 
Letting $L_{\varphi'}$ denote the Lipschitz constant of $\varphi'$ on $\mathcal Z$, and gathering the above results gives  
\[
\sup_{x \in X} \left| \varphi\circ {u}^{\rm u}(x) - f^{\rm u}(x)\right| \leq \max_{k\in\{1,\ldots,N\}}\frac{3L_{\varphi'}}{2}\wid{X_k}^2 \leq \frac{3L_{\varphi'}C_x^2}{2}\left(\frac{\wid{X}}{N_0}\right)^2 \, .
\]
Combining this bound with that in \eqref{eq:unicvxunder2}, we obtain the following bound of the right-hand side of~\eqref{eq:unicvxunder}
\[
\sup_{x \in X} \left| \varphi \circ {{u}}(x) - \varphi\circ {u}^{\rm u}(x) \right| + \sup_{x \in X} \left| \varphi\circ {u}^{\rm u}(x) - f^{\rm u}(x)\right| \leq \left(L_\varphi C_u +\frac{3L_{\varphi'}C_x^2}{2}\right) \left(\frac{\wid{X}}{N_0}\right)^2 \, .
\]
Finally, noting that the Lipschitz constants $L_{\varphi}$ and  $L_{\varphi'}$ simply depend on $\mathcal Z$ rather than $X$ and $N_0$, and by the assumption that the constants $C_x$ and $C_u$ are independent of $X$ and $N_0$, we conclude that the constant
$C_f \coloneqq L_\varphi C_u +\frac{3L_{\varphi'}C_x^2}{2}$ is again independent of $X$ and $N_0$. 
\end{proof}

Finally, it is clear that any variable can be encoded exactly by a continuous piecewise-linear function. In analogy with classical approximation theory, where any continuous univariate function can be uniformly approximated by a sequence of continuous piecewise-linear functions with increasingly fine partitions, the recursive application of Algorithms~\ref{alg:pwl_addition}--\ref{alg:pwl_composition} can therefore construct pairs of continuous piecewise-linear functions that bracket sufficiently regular univariate functions within any desired finite precision (property P3). The key enabling property is that Algorithms~\ref{alg:pwl_addition}--\ref{alg:pwl_composition} only retain or subdivide existing segments, but never coarsen them---see Remark~\ref{rmk:never_coarsen}. Hence, every newly generated segment remains bounded by $C_x\frac{\wid{X}}{N_0}$, and the segment-width hypothesis required for the next recursive application follows immediately.

In practice, partitioning the domain of the variables thus provides a means of reducing the estimation error introduced by the composition operations. But this refinement also creates a tradeoff with the computational complexity of a superposition relaxation, which we illustrate in the following case studies.

\section{Computational Studies}
\label{sec:casestudies}

In this section, the construction and convergence of superposition relaxations are studied numerically for a selection of factorable functions. These relaxations are furthermore compared with McCormick relaxations \citep{McCormick1976}, including their generalization to enable product, minimum and maximum operations using multivariate composition per \cite{Tsoukalas2014}.

All the computations that led to these results are performed with the open-source library {\sf MC\textsuperscript{++}} (version 5.0) \citep{MCPP} which features C\textsuperscript{++} classes for the evaluation of superposition relaxations for factorable expressions, alongside McCormick relaxations and several other arithmetics. These arithmetics can also be accessed through the Python-binding interface {\sf PyMC\textsuperscript{++}}.\footnote{{\sf PyMC\textsuperscript{++}} is available on {\sf PyPI} under the name {\sf pymcpp} (\url{https://pypi.org/project/pymcpp/}).} For higher flexibility, the superposition relaxation arithmetic was implemented in a templated C\textsuperscript{++} class, with the template type specializing to various parameterizations and operations on the univariate summands, including piecewise-constant and continuous piecewise-linear parameterizations. All of these arithmetics can be applied either on-the-fly through operator overloading, or via the construction of a computational graph (DAG) and its subsequent evaluation to take advantage of common subexpressions. Note also that any round-off caused by operations between the univariate summands in a superposition relaxation, or to evaluate the convex underestimator, concave overestimator and their subgradients in a McCormick relaxation, are not accounted for in the current implementation.

\subsection{Case Study 1}

Let $\mathcal{X}=[-1,2]\times[-2,1]$ and consider the factorable function $f:\mathcal{X}\to\mathbb{R}$ with
\begin{align}
\label{eq:CS1}
f(\bm x) \coloneqq x_1 x_2 \left( x_1 \left(\exp(x_1)-\exp(-x_1)\right) + x_2 \left(\exp(x_2)-\exp(-x_2)\right) \right).
\end{align}
Although it is not separable, notice that \eqref{eq:CS1} presents some separable structures in the univariate terms  $x_1 \left(\exp(x_1)-\exp(-x_1)\right)$ and $x_2 \left(\exp(x_2)-\exp(-x_2)\right)$.

\begin{figure}[tbp]
\centering
\includegraphics[width=0.98\linewidth]{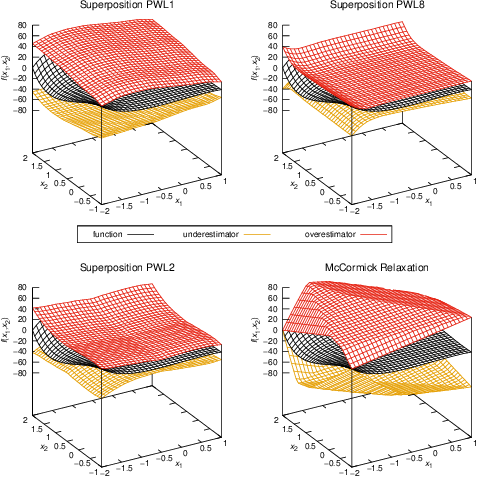}
\caption{Comparison between relaxations of the test function \eqref{eq:CS1}. The superposition relaxations are obtained with piecewise-linear univariate estimators initialized with $N_{\rm ini}=1$ (top-left), $2$ (bottom-left) or $8$ (top-right) equidistant segments for the variables $x_1$ and $x_2$. The McCormick relaxation (bottom-right) involves recomputing the values of both convex underestimator and concave overestimator on a grid.}
\label{fig:CS1_plot}
\end{figure}

Figure~\ref{fig:CS1_plot} compares the superposition relaxations of the test function \eqref{eq:CS1} computed with continuous piecewise-linear summands and initializing the two variables $x_1$ and $x_2$ over an equidistant partition with $N_{\rm ini}=1$, $2$ or $8$ segments. Increasing the initial partition size of both variables noticeably tightens both superposition under- and overestimators, with little further improvement observed beyond 8 segments. This improved tightness comes at the cost of added complexity for the four univariate summands, each containing 9 segments for $N_{\rm ini}=1$, increasing to 16 segments for $N_{\rm ini}=2$, and to 61 segments for $N_{\rm ini}=8$. Recall that these extra segments/breakpoints are introduced adaptively by Algorithm~\ref{alg:pwl_composition} when composing the piecewise-linear summands with nonlinear terms, both exponential and square power terms in this case. Figure~\ref{fig:CS1_plot} also shows the McCormick relaxation of \eqref{eq:CS1}, which uses the multivariate composition rule for product operations here. Although the convex underestimator and concave overestimator appear to be exact at several edges of the domain $\mathcal{X}$, these estimators are also found to be significantly weaker than their superposition counterparts over most of the domain $\mathcal{X}$. A clear advantage of superposition relaxations here lies in their superior ability to describe the variations and nonconvexities present in \eqref{eq:CS1}, enabled by the flexibility of the univariate summand parameterization.

\begin{figure}[tbp]
\centering
\includegraphics[width=0.98\linewidth]{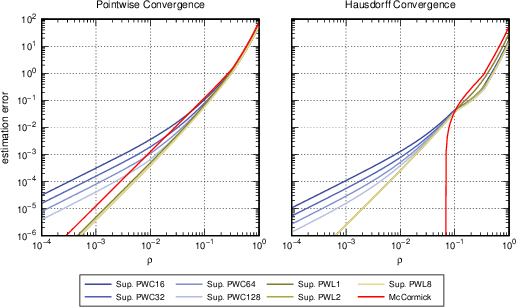}
\caption{Comparison between the convergence rate of superposition relaxations and McCormick relaxations for the test function \eqref{eq:CS1}, in the pointwise sense (left) and Hausdorff sense (right). The superposition relaxations with piecewise-constant (PWC) univariate estimators consider equal subdivisions on $16$, $32$, $64$ or $128$ segments for the variables $x_1$ and $x_2$. The superposition relaxations with piecewise-linear (PWL) univariate summands are initialized with $N_{\rm ini}=1$, $2$ or $8$ equidistant segments for the variables $x_1$ and $x_2$.}
\label{fig:CS1_rate}
\end{figure}

A convergence analysis for various relaxations of the test function \eqref{eq:CS1} is presented in Figure~\ref{fig:CS1_rate}, where the log-log convergence plots report the estimation error for a varying contraction ratio $\rho\in(0,1]$, corresponding to the reduced domain $\bm X_\rho \coloneqq \rho\mathcal{X} + (1-\rho)\bm x^\ast$ centered around the mid-point $\bm x^\ast=[-0.5\ 0.5]^\intercal$. The same three superposition relaxations with continuous piecewise-linear summands for $N_{\rm ini}=1$, $2$ and $8$ present a very similar convergence behavior, all showing the expected quadratic convergence order in both the pointwise (left plot) and Hausdorff (right plot) sense. The comparison with McCormick relaxations confirms that superposition relaxations remain tighter on smaller domains in the pointwise view; while McCormick relaxations eventually regain the advantage in the Hausdorff sense on smaller domains, as they estimate the range of \eqref{eq:CS1} exactly for $\rho\lesssim 7\times 10^{-2}$. Also displayed on Figure~\ref{fig:CS1_rate} for comparison, are the superposition relaxations with piecewise-constant summands on fixed variable grids with $N_{\rm grid}=16$, $32$, $64$ and $128$ segments. These relaxations behave similarly to their piecewise-linear counterparts on wider domains, but their convergence rates start departing from quadratic order below $\rho\lesssim 10^{-2}$ where they become linear only. Nevertheless, increasing the partition size for the two variables provides a means of delaying this transition to linear order.

\subsection{Case Study 2}

Consider the function $f:\mathcal{X}\subset\mathbb{R}^n\to\mathbb{R}$ defined on the hypercube $\mathcal{X}=[0,10]^n$ by
\begin{align}
\label{eq:CS2}
f(\bm x) \coloneqq\ & \sum_{k=1}^m \frac{1}{\displaystyle\sum_{i=1}^n (x_i-\alpha_{i,k})^2+\beta_k},
\end{align}
commonly known as Shekel's foxholes in global optimization benchmarks \citep{Naser2025}. This function has $m$ local maxima in $\mathbb{R}^n$ and a unique global maximum at $x_i^\ast=4$, e.g., for $m=10$, $n=2p$ with $p>0$, and the following parameter values for $0<j\leq p$,
\begin{align*}
{\bm\alpha}_{2j} =\ & [4\ 1\ 8\ 6\ 3\ 2\ 5\ 8\ 6\ 7]^\intercal\\
{\bm\alpha}_{2j-1} =\ & [4\ 1\ 8\ 6\ 7\ 9\ 3\ 1\ 2\ 3.6]^\intercal\\
{\bm\beta} =\ & [0.1\ 0.2\ 0.2\ 0.4\ 0.4\ 0.6\ 0.3\ 0.7\ 0.5\ 0.5]^\intercal.
\end{align*}
Like \eqref{eq:CS1} before, the function \eqref{eq:CS2} is not separable but it presents a clear separable structure in the denominator $\sum_{i=1}^n (x_i-\alpha_{i,k})^2$. This case is also useful to analyze the behavior of superposition relaxations in higher dimensions.

\begin{figure}[tbp]
\centering
\includegraphics[width=0.98\linewidth]{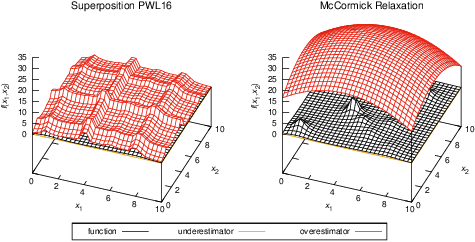}
\caption{Comparison between relaxations of the test function \eqref{eq:CS2}. The superposition relaxation is obtained with piecewise-linear univariate estimators initialized with $N_{\rm ini}=16$ (left). The McCormick relaxation (right) involves recomputing the values of both convex underestimator and concave overestimator on a grid.}
\label{fig:CS2_plot}
\end{figure}

The left plot on Figure~\ref{fig:CS2_plot} shows the superposition relaxation of \eqref{eq:CS2} for $n=2$ dimensions, computed with continuous piecewise-linear summands and initialized with an equidistant partition over $N_{\rm ini}=16$ segments. This relaxation captures well the main peaks in the Shenkel function, including the global maximum at ${\bm x}^\ast=[4\ 4]^\intercal$. In particular, it is significantly tighter than its McCormick relaxation counterpart as shown on the right plot. However, because of its separable nature, the superposition relaxation also presents extra (parasitic) peaks, for instance around ${\bm x}=[1\ 4]^\intercal$ and ${\bm x}=[4\ 1]^\intercal$ where the Shenkel function does not have such peaks (compare with the plot on the right). With regard to complexity, the two univariate underestimator summands comprise 179 segments and the overestimator summands 46--50 segments for $N_{\rm ini}=16$, which are added adaptively when composing the piecewise-linear summands with inverse and square power terms using Algorithm~\ref{alg:pwl_composition}. Interestingly, with $n=6$ dimensions (and again $N_{\rm ini}=16$), the numbers of segments in the six univariate underestimator/overestimator summands remain identical to the $n=2$ case.

\begin{figure}[tbp]
\centering
\includegraphics[width=0.98\linewidth]{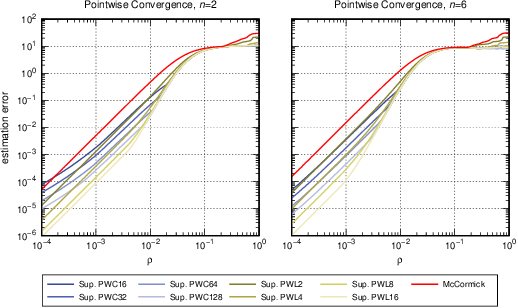}
\caption{Comparison between the convergence rate of superposition relaxations and McCormick relaxations for the test function \eqref{eq:CS2}, in the pointwise sense, for $n=2$ (left) and $n=6$ (right) dimensions. The superposition relaxations with piecewise-constant (PWC) univariate estimators consider equal subdivisions on $16$, $32$, $64$ or $128$ segments for the variables $x_i$. The superposition relaxations with piecewise-linear (PWL) univariate summands are initialized with $N_{\rm ini}=2$, $4$, $8$ or $16$ equidistant segments for the variables $x_i$.}
\label{fig:CS2_rate}
\end{figure}

The convergence analysis in Figure~\ref{fig:CS2_rate} for various relaxations of the test function \eqref{eq:CS2} contracts the variable domain around the global maximum $\bm x^\ast$. The general trends are similar for either $n=2$ or $6$ dimensions, distinguishing three main parts. In the central part around $\rho\approx 10^{-1}$, the pointwise estimation error is comparable amongst the various relaxation schemes. In the part with smaller subdomains $\rho\ll 10^{-1}$, the superposition relaxations significantly reduce the pointwise estimation error compared to their McCormick counterparts, mainly due to the ability of the superposition under- and overestimators to describe nonconvex functions. In particular, the estimation error benefits significantly from refining the variable partition both with piecewise-constant (e.g. PWC128) or continuous piecewise-linear (e.g. PWL16) summands, although the former ultimately transition to a linear convergence rate while the latter -- and the McCormick relaxations too -- are quadratically convergent since \eqref{eq:CS2} is sufficiently smooth. In the part with larger subdomains $\rho\gg 10^{-1}$, the superposition relaxations again yield tighter relaxations, especially those with larger partitions for which the estimation error remains about constant---a very desirable property for a relaxation scheme to have on wider domains.

\subsection{Case Study 3}

Let $\mathcal{X}=[-3,3]^2$ and consider the function $f:\mathcal{X}\to\mathbb{R}$ with
\begin{align}
\label{eq:CS3}
f(\bm x) \coloneqq\ & 3(1-x_1)^2\exp\left(-x_1^2-(x_2+1)^2\right)-10\left(\frac{x_1}{5}-x_1^3-x_2^5\right)\exp\left(-x_1^2-x_2^2\right) \nonumber\\ & -\frac{1}{3}\exp\left(-(x_1+1)^2-x_2^2\right),
\end{align}
often referred to as {\sf MATLAB}'s peak function. Unlike \eqref{eq:CS1} and \eqref{eq:CS2}, this factorable function does not present obvious separable structure.

\begin{figure}[tbp]
\centering
\includegraphics[width=0.98\linewidth]{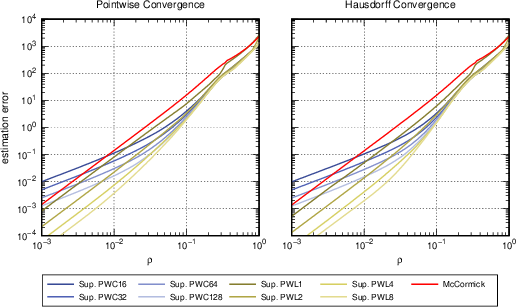}
\caption{Comparison between the convergence rate of superposition relaxations and McCormick relaxations for the test function \eqref{eq:CS3}, in the pointwise sense (left) and Hausdorff sense (right). The superposition relaxations with piecewise-constant (PWC) univariate estimators consider equal subdivisions on $16$, $32$, $64$ or $128$ segments for the variables $x_1$ and $x_2$. The superposition relaxations with piecewise-linear (PWL) univariate summands are initialized with $N_{\rm ini}=1$, $2$, $4$ or $8$ equidistant segments for the variables $x_1$ and $x_2$.}
\label{fig:CS3a_rate}
\end{figure}

Both the superposition and McCormick relaxations over the full domain $\mathcal{X}$ largely overestimate the actual function range $[-6.55,8.11]$, with estimation errors as large as 1300--2600. This large overestimation is attributed to the complex algebraic expressions in \eqref{eq:CS3}, leading to a dependency problem. However, the overestimation decreases quickly as the variable domain shrinks. Figure~\ref{fig:CS3a_rate} compares the convergence rates of various relaxations as the domain $\mathcal{X}$ is contracted around the global maximum of \eqref{eq:CS3} at $\bm x^* \approx [-0.0106\ 1.5803]^\intercal$---a similar convergence behavior is observed around other extrema of \eqref{eq:CS3}. All four superposition relaxations with continuous piecewise-linear summands enjoy a quadratic convergence order in both the pointwise (left plot) and Hausdorff (right plot) sense. The effect of partitioning each variable domain over $N_{\rm ini}=1$, $2$, $4$ and $8$ equidistant segments is found to be significant, especially over smaller domains where the overestimation eror decreases by 10 fold or more between $N_{\rm ini}=1$ and $8$. The superposition relaxations with piecewise constant sumands behave similarly to their piecewise linear counterparts on large domains, then start diverging below $\rho\lesssim 10^{-1}$. In particular, increasing the (fixed) partition size of the piecewise constant sumands -- here from $N_{\rm grid}=16$ to $128$ segments -- is effective at delaying the onset of linear convergence. Lastly, the comparison with McCormick relaxations shows that superposition relaxations are consistently tighter, in both the pointwise and Hausdorff sense, which is remarkable given the lack of separable structures in \eqref{eq:CS3}.
 

\begin{figure}[tbp]
\centering
\includegraphics[width=0.98\linewidth]{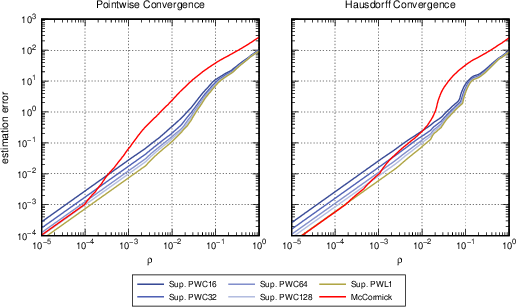}
\caption{Comparison between the convergence rate of superposition relaxations and McCormick relaxations for an ANN with 4 hidden layers of 40 neurons each and ReLU activation functions describing the test function \eqref{eq:CS3}, in the pointwise sense (left) and Hausdorff sense (right). The superposition relaxations with piecewise-constant univariate estimators consider equal subdivisions on $16$, $32$, $64$ or $128$ segments for the variables $x_1$ and $x_2$. The superposition relaxations with piecewise-linear univariate summands are initialized with a single segment for the variables $x_1$ and $x_2$.}
\label{fig:CS3b_rate}
\end{figure}

Next, we conduct another convergence analysis, this time on an artificial neuron network (ANN) surrogate of the test function \eqref{eq:CS3}. The selected ANN architecture comprises 4 hidden layers of 40 neurons each with ReLU activation functions in each layer, and the ANN was trained using {\sf scikit-learn} with samples of \eqref{eq:CS3} on a $600\times 600$ grid of the domain $\mathcal{X}$ and the mean-squared error (with an $\ell_1$ regularization) as loss function. Note that the superposition relaxations with continuous piecewise-linear summands are initialized without extra breakpoints ($N_{\rm ini}=1$) in this case, as the current implementation in {\sf MC\textsuperscript{++}} can automatically add the relevant breakpoints to describe the maximum (ReLU) operation between a piecewise-linear function exactly and a constant exactly; that is, initializing the variables with extra breakpoints would not improve the relaxations. Figure~\ref{fig:CS3b_rate} again compares the convergence rates of various relaxations as the domain $\mathcal{X}$ is contracted around the global maximum of the ANN surrogate of \eqref{eq:CS3} at $\bm x^* \approx [-0.0150\ 1.6176]^\intercal$. The log-log convergence plots confirm the loss of quadratic convergence order both in the pointwise (left) and Hausdorff (right) sense, due to the nondifferentiability of the max operations. The superposition relaxations with piecewise-constant summands behave similarly and approach their piecewise-linear counterpart as the grid partition of the variables $x_1$ and $x_2$ is refined. Lastly, the estimation error of the McCormick relaxation is more than double that of the superposition relaxations on the full domain $\mathcal{X}$, a difference that first increases as the domain arod the ANN maximum is reduced, before starting to close in with the superposition relaxations below $\rho\lesssim 10^{-2}$.

\begin{table}[tb]
\caption{Comparison of computational times for superposition relaxations and McCormick relaxations for the test function \eqref{eq:CS3} and its ANN surrogate, with wall-times in microseconds.\textsuperscript{$\dagger$}}
\label{tab:CS3_walltime}
\small
\setlength\tabcolsep{0pt}
\begin{tabular*}{\linewidth}{@{\extracolsep{\fill}}lrr}
\toprule
Relaxation & Test function \eqref{eq:CS3} & ANN surrogate of \eqref{eq:CS3}\\
\midrule
McCormick                 &   1 µs &   37 µs\\
McCormick \& subgradient  &   3 µs &  181 µs\\[.35em]
Superposition PWC16       &  31 µs & 1390 µs\\
Superposition PWC32       &  52 µs & 1720 µs\\
Superposition PWC64       &  98 µs & 2550 µs\\
Superposition PWC128      & 193 µs & 4030 µs\\[.35em]
Superposition PWL1        &  16 µs & 5750 µs\\
Superposition PWL2        &  18 µs & --\\
Superposition PWL4        &  24 µs & --\\
Superposition PWL8        &  38 µs & --\\
\bottomrule
\end{tabular*}
\footnotesize\textsuperscript{$\dagger$}{Lenovo ThinkPad X1 Carbon Gen 10 with 12\textsuperscript{th} Gen Intel\textsuperscript{®} Core™ i7-1260P × 16, 32.0 GiB memory, Ubuntu 22.04 operating system}
\end{table}

A computational performance comparison between various relaxations is reported in Table~\ref{tab:CS3_walltime}. Overall, relaxing the ANN surrogate of \eqref{eq:CS3} is far more computationally demanding than relaxing the test function \eqref{eq:CS3} itself as the ANN entails a much greater operation count. The McCormick relaxations are also significantly faster to compute than any of the superposition relaxations, even when subgradients of the McCormick relaxations are propagated. Although it is worth recalling that a McCormick relaxation computes the convex underestimator and concave overestimator values as well as the respective subgradients at one point $\bm x\in\mathcal{X}$ only -- here the global maximum -- while a superposition relaxation computes parameterized univariate summands for the superposition under- and overestimator that remain valid over the whole domain $\mathcal{X}$. 

For the test function \eqref{eq:CS3}, the superposition relaxations with piecewise-linear (PWL) summands are faster to compute than their piecewise-constant (PWC) counterparts. For instance, the piecewise-linear summands in a superposition relaxation where the variables $x_1$ and $x_2$ are partitioned into 8 segments (PWL8) comprise between 37--57 segments, but this relaxation is still faster to compute than with piecewise-constant summands on an equipartition with 32 segments (PWC32), thus confirming that adaptive partitioning can offer computational benefits. The performance for the relaxation of the ANN surrogate shows a different trend. Even with a partition size of 128 the superposition relaxations with piecewise-constant summands (PWC128) is faster to compute than its piecewise-linear counterpart (PWC1), and this despite the under- and overestimator summands comprising few segments, between 74--81. In future work it would be relevant, therefore, to implement segment-reduction strategies for the piecewise-linear summands in order to compute superposition relaxations more efficiently. 



More broadly, a natural direction for future research lies in systematically investigating the performance of superposition arithmetic in global optimization algorithms, particularly within spatial branch-and-bound frameworks and reduced-space problem formulations \citep{Eperly1997,Bongartz2019}. In this context, it is crucial to evaluate how the tighter relaxations affect node pruning efficiency, bound tightening, and convergence behavior when compared to McCormick relaxations and other state-of-the-art set arithmetics. Emphasis should also be placed on high-dimensional optimization problems, where the separable structure of superposition relaxations may alleviate some aspects of the curse of dimensionality. Additionally, hybrid strategies that combine superposition relaxations with more computationally efficient relaxations in a problem-dependent and dynamic fashion may offer a promising path to balancing relaxation quality against computational overhead.

\section{Conclusion and Future Work}

The superposition relaxation arithmetic introduced through this paper constructs separable under- and overestimators that bracket the graph of a multivariate factorable function over compact domains. Maintaining separability of the estimators enables efficient range computation and estimator propagation with linear complexity in the number of variables, while still capturing global nonlinearities and nonconvexities in the original function. A key technical contribution is the ability to propagate these separable estimators through compositions with general nonlinear univariate terms by exploiting ridge function structures in combination with global monotonicity and convexity properties. We established that superposition relaxations can preserve quadratic convergence in the pointwise and Hausdorff sense under mild regularity assumptions, e.g. continuous-differentiability with Lipschitz continuous derivatives of outer functions. We furthermore introduced parameterizations of the univariate summands, either as piecewise-constant or continuous piecewise-linear functions, that are computationally tractable, preserve affine invariance, and are uniformly approximating. Our numerical experiments show that, in comparison to McCormick relaxations, superposition relaxations consistently produce tighter enclosures, even in challenging cases without clear separable structures in the estimated function or with a large number of operations as in artificial neural networks. However, this improved tightness comes at a higher computational cost, primarily due to the increased parameterization complexity. Future research, therefore, needs to conduct a systematic investigation of the performance of superposition relaxations in global optimization, for instance as part of a spatial branch-and-bound algorithm. Finally, while this work focused on superposition arithmetic with separable under- and overestimators, it will be interesting to investigate higher-order superposition relaxations, where the summands depend on two or more variables to further reduce the relaxation gap.

\backmatter

\bmhead{Data Statement}

An implementation of superposition arithmetic is available through the C\textsuperscript{++} library {\sf MC\textsuperscript{++}} and its Python-binding interface {\sf PyMC\textsuperscript{++}} under \url{https://github.com/omega-icl/mcpp/tree/version-5.0} \citep{MCPP}. {\sf PyMC\textsuperscript{++}} is available as a standalone package on {\sf PyPI} under \url{https://pypi.org/project/pymcpp/}. This code is used in the computational studies in Section~\ref{sec:casestudies}.

\bmhead{Funding Declaration}

This research was supported by the Engineering and Physical Sciences Research Council (EPSRC) of UK Research and Innovation (UKRI) under grant EP/W003317/1.

\bmhead{Acknowledgements} 

The authors are grateful to the handling editor and to the two anonymous reviewers for their careful reading of the manuscript and for their constructive comments.

\begin{appendices}

\section{Are the Superposition Relaxations of Convex Ridge Functions Optimal in Some Sense?}\label{app:optimality}

The superposition relaxation constructions for the ridge function \eqref{eq:cvxridge} with a convex profile function $\varphi$ were established in Propositions~\ref{prop:cvxridge_over} and \ref{prop:cvxridge_under}, but it remains unclear whether the summands \eqref{eq:monotonic_underest}, \eqref{eq:monotonic_overest} and \eqref{eq:nonmonotonic_underest} are optimal in some sense. In this appendix, we conduct an investigation using numerical optimization.

Starting with the superposition underestimator $f^{\rm u}$, we consider the $L_p$ error minimization problem 
\begin{align}\label{eq:fuopt}
\min_{f^{\rm u}_i\in\mathcal{C}^0(X_i)} & \int_{\bm X} \left| \varphi\left(\sum_{i=1}^n x_i\right) - \sum_{i=1}^n f_i^{\rm u}(x_i) \right|^p d{\bm x}\\
\text{s.t.}\ \ & \sum_{i=1}^n f_i^{\rm u}(x_i) \leq \varphi\left(\sum_{i=1}^n x_i\right), \quad \forall{\bm x}\in{\bm X}\nonumber\\
& \sum_{i=1}^n f_i^{\rm u}(x_i) \geq \min_{{\bm x}\in{\bm X}} \varphi\left(\sum_{i=1}^n x_i\right), \quad \forall{\bm x}\in{\bm X},\nonumber
\end{align}
where the constraints guarantee that the underestimator is valid and tracks the minimum of the ridge function. For simplicity, we discretize this problem on a grid with $N$ points along each coordinate,
\begin{align}\label{eq:fuopt_discr}
\min_{f^{\rm u}_{i,{\bm\kappa}}} & \sum_{{\bm\kappa}\in\{1,\ldots,N\}^n} \frac{1}{N^{\frac{n}{p}}}\left| \varphi\left(\sum_{i=1}^n x_{i,\kappa_i}\right) - \sum_{i=1}^n f_{i,\kappa_i}^{\rm u} \right|^p\\
\text{s.t.}\ \ & \sum_{i=1}^n f_{i,\kappa_i}^{\rm u} \leq \sum_{{\bm\kappa}\in\{1,\ldots,N\}^n} \varphi\left(\sum_{i=1}^n x_{i,\kappa_i}\right),\quad \forall i\in\{1,\ldots,n\},\ {\bm\kappa}\in\{1,\ldots,N\}^n\nonumber\\
& \sum_{i=1}^n f_{i,\kappa_i}^{\rm u} \geq \min_{{\bm x}\in{\bm X}} \varphi\left(\sum_{i=1}^n x_i\right),\quad \forall{\bm\kappa}\in\{1,\ldots,N\}^n.\nonumber
\end{align}
We solve the convex NLP \eqref{eq:fuopt_discr} using {\sf Gurobi} solver from {\sf GAMS} (v49.6.0). The results obtained for selected convex profile functions $\varphi$, nondecreasing on various 2-d domains $\bm X$, are summarized in Table~\ref{tab:fuopt}. The focus is on the $L_1$ and $L_2$ errors for their computational tractability. The $L_\infty$ error was also considered but these results are not reported in Table~\ref{tab:fuopt} as the corresponding NLP \eqref{eq:fuopt_discr} exhibits multiple global optima.

In all the cases in Table~\ref{tab:fuopt}, the minimal (discretized) $L_1$ and $L_2$ errors for these optimal underestimators are identical, within numerical accuracy, to the (discretized) $L_1$ and $L_2$ errors of the superposition underestimator $f^{\rm u}$ with summands as in \eqref{eq:monotonic_underest}. This comparison suggests that the summands \eqref{eq:monotonic_underest} may indeed provide an optimal underestimator, both in the $L_1$ and $L_2$ sense, among all possible separable underestimators of the convex increasing ridge function \eqref{eq:cvxridge} that track its minimum exactly. In general, tighter separable underestimators could be obtained by omitting the second constraint in \eqref{eq:fuopt_discr}, although these underestimators would then lose the minimum-tracking property; see Remark~\ref{rmk:cvxridge_under}. A similar comparison is given in Table~\ref{tab:fuopt2} for selected nonmonotonic convex profile functions $\varphi$ on $\bm X$. The results suggest that the univariate summands \eqref{eq:nonmonotonic_underest}, constructed from the decomposition \eqref{eq:monotonic_decomp} to maintain their minimum-tracking property, may no longer provide optimal underestimators in the $L_1$ or $L_2$ sense, since (slightly) tighter separable underestimators can be computed from solving \eqref{eq:fuopt_discr}. 

\begin{table}[tb]
\caption{Comparison between the minimal (discretized) $L_1$ and $L_2$ errors computed via \eqref{eq:fuopt_discr} and those of the superposition underestimator constructed via \eqref{eq:monotonic_underest} for a convex nondecreasing profile function $\varphi$. The discretization uses $N=120$ equidistant grid points for each coordinate.}
\label{tab:fuopt}
\small
\setlength\tabcolsep{0pt}
\begin{tabular*}{\linewidth}{@{\extracolsep{\fill}}llrrrr}
\toprule
Profile,          & Domain,                  & \multicolumn{2}{c}{$L_1$ error}     & \multicolumn{2}{c}{$L_2$ error}\\
\cmidrule{3-6}
$\varphi(z)$      & $\bm X$                  & Eq.~\eqref{eq:fuopt_discr} & Eq.~\eqref{eq:monotonic_underest} & Eq.~\eqref{eq:fuopt_discr} & Eq.~\eqref{eq:monotonic_underest} \\
\midrule
$z^2$        & $[0.4,2]^2$                   &  1.280000 & {\color{Green}1.280000} & 1.713778 & {\color{Green}1.713778}\\
$z^2$        & $[0,1]\times[0.4,2]$          &  0.800000 & {\color{Green}0.800000} & 1.071111 & {\color{Green}1.071111}\\
$\exp(z)$    & $[-1,1]^2$                    &  0.656730 & {\color{Green}0.656730} & 1.098116 & {\color{Green}1.098116}\\
$\exp(z)$    & $[-1,0]\times[1,3]$           &  1.584865 & {\color{Green}1.584865} & 2.489730 & {\color{Green}2.489730}\\
$-\sqrt{-z}$ & $[-1,-0.2]^2$                 &  0.024369 & {\color{Green}0.024369} & 0.035808 & {\color{Green}0.035808}\\
$-\sqrt{-z}$ & $[-2,-0.5]\times[-1,-0.2]$    &  0.024258 & {\color{Green}0.024258} & 0.035436 & {\color{Green}0.035436}\\
$-\log(-z)$  & $[-0.8,-0.2]\times[-2,-1]$    &  0.030759 & {\color{Green}0.030759} & 0.044287 & {\color{Green}0.044287}\\
$-\log(-z)$  & $[-1,-0.05]\times[-0.7,-0.2]$ &  0.096174 & {\color{Green}0.096174} & 0.151491 & {\color{Green}0.151491}\\
\bottomrule
\end{tabular*}
\end{table}

\begin{table}[tb]
\caption{Comparison between the minimal (discretized) $L_1$ and $L_2$ errors computed via \eqref{eq:fuopt_discr} and those of the superposition underestimator constructed via \eqref{eq:nonmonotonic_underest} for a convex nonmonotonic profile function $\varphi$. The discretization uses $N=120$ equidistant grid points for each coordinate.}
\label{tab:fuopt2}
\small
\setlength\tabcolsep{0pt}
\begin{tabular*}{\linewidth}{@{\extracolsep{\fill}}llrrrr}
\toprule
Profile,          & Domain,                  & \multicolumn{2}{c}{$L_1$ error}     & \multicolumn{2}{c}{$L_2$ error}\\
\cmidrule{3-6}
$\varphi(z)$      & $\bm X$                  & Eq.~\eqref{eq:fuopt_discr} & Eq.~\eqref{eq:nonmonotonic_underest} & Eq.~\eqref{eq:fuopt_discr} & Eq.~\eqref{eq:nonmonotonic_underest} \\
\midrule
$z^2$        & $[-1,2]^2$                &  2.296281 & {\color{Green}2.296281} & 3.632748 & {\color{Green}3.632748}\\
$z^2$        & $[-1,2]\times[0.5,2]$     &  2.043388 & {\color{Red}2.043414}   & 2.832724 & {\color{Red}2.832719} \\
$|z|$        & $[-1,2]^2$                &  0.963595 & {\color{Green}0.963595} & 1.156356 & {\color{Green}1.156356} \\
$|z|$        & $[-1,2]\times[0.5,2]$     &  0.379995 & {\color{Red}0.382061}   & 0.414386 & {\color{Red}0.415531} \\
$\exp(z)-z$  & $[-0.5,1]^2$              &  0.455369 & {\color{Green}0.455369} & 0.814373 & {\color{Green}0.814373} \\
$\exp(z)-z$  & $[-0.5,1]\times[-1,-0.5]$ &  0.084208 & {\color{Red}0.084214}   & 0.119151 & {\color{Red}0.119155} \\
\bottomrule
\end{tabular*}
\end{table}

Similarly for the superposition overestimator $f^{\rm o}$ with summands as in \eqref{eq:monotonic_overest}, we consider the $L_p$ error minimization problem 
\begin{align}\label{eq:foopt}
\min_{f^{\rm u}_i\in\mathcal{C}^0(X_i)} & \int_{\bm X} \left| \varphi\left(\sum_{i=1}^n x_i\right) - \sum_{i=1}^n f_i^{\rm o}(x_i) \right|^p d{\bm x}\\
\text{s.t.}\ \ & \sum_{i=1}^n f_i^{\rm o}(x_i) \geq \varphi\left(\sum_{i=1}^n x_i\right), \quad \forall{\bm x}\in{\bm X},\nonumber
\end{align}
which we again discretize on a grid with $N$ points along each coordinate,
\begin{align}\label{eq:foopt_discr}
\min_{f^{\rm o}_{i,{\bm\kappa}}} & \sum_{{\bm\kappa}\in\{1,\ldots,N\}^n} \frac{1}{N^{\frac{n}{p}}}\left| \varphi\left(\sum_{i=1}^n x_{i,\kappa_i}\right) - \sum_{i=1}^n f_{i,\kappa_i}^{\rm o} \right|^p\\
\text{s.t.}\ \ & \sum_{i=1}^n f_{i,\kappa_i}^{\rm o} \geq \sum_{{\bm\kappa}\in\{1,\ldots,N\}^n} \varphi\left(\sum_{i=1}^n x_{i,\kappa_i}\right),\quad \forall i\in\{1,\ldots,n\},\ {\bm\kappa}\in\{1,\ldots,N\}^n.\nonumber
\end{align}
Notice that we do not include a constraint for the overestimator to track the maximum of the ridge function exactly, as the optimal overestimators computed via \eqref{eq:foopt_discr} seem to automatically satisfy this condition with a convex profile function $\varphi$. Table~\ref{tab:foopt} summarizes the results obtained by solving the convex NLP \eqref{eq:foopt_discr} using {\sf Gurobi} for selected convex profile functions $\varphi$, that are either nondecreasing or nonmonotonic on various 2-d domains $\bm X$. In all the cases, the (discretized) $L_1$ error of the superposition overestimator $f^{\rm o}$ constructed with the univariate summands \eqref{eq:monotonic_overest} matches, within numerical precision, the minimal $L_1$ error for the optimal separable overestimator computed via \eqref{eq:foopt_discr}. This comparison thus suggests that the summands \eqref{eq:monotonic_overest} may indeed determine an optimal separable overestimator in the $L_1$ sense when $\varphi$ is convex. In contrast, the summands \eqref{eq:monotonic_overest} may not be optimal in the $L_2$ sense, as separable overestimators with a lower $L_2$ discretization error can be computed from solving \eqref{eq:foopt_discr} in general. 

\begin{table}[tb]
\caption{Comparison between the minimal (discretized) $L_1$ and $L_2$ errors computed via \eqref{eq:foopt_discr} and those of the superposition overestimator constructed via \eqref{eq:monotonic_overest} for a convex nondecreasing (top part) or nonmonotonic (bottom part) profile function $\varphi$. The discretization uses $N=120$ equidistant grid points for each coordinate.}
\label{tab:foopt}
\small
\setlength\tabcolsep{0pt}
\begin{tabular*}{\linewidth}{@{\extracolsep{\fill}}llrrrr}
\toprule
Profile,           & Domain,                 & \multicolumn{2}{c}{$L_1$ error} & \multicolumn{2}{c}{$L_2$ error}\\
\cmidrule{3-6}
$\varphi(z)$       & ${\bm X}$               & Eq.~\eqref{eq:foopt_discr} & Eq.~\eqref{eq:monotonic_overest} & Eq.~\eqref{eq:foopt_discr} & Eq.~\eqref{eq:monotonic_overest} \\
\midrule
$z^2$        & $[0.4,2]^2$                   &  0.433778 & {\color{Green}0.433778} & 0.671994 & {\color{Green}0.671994}\\
$z^2$        & $[0,1]\times[0.4,2]$          &  0.271111 & {\color{Green}0.271111} & 0.419996 & {\color{Green}0.419996}\\
$\exp(z)$    & $[-1,1]^2$                    &  0.441386 & {\color{Green}0.441386} & 0.719049 & {\color{Green}0.719049}\\
$\exp(z)$    & $[-1,0]\times[1,3]$           &  0.889541 & {\color{Green}0.889541} & 1.426867 & {\color{Red}1.432228}\\
$-\sqrt{-z}$ & $[-1,-0.2]^2$                 &  0.011773 & {\color{Green}0.011773} & 0.018510 & {\color{Green}0.018510}\\
$-\sqrt{-z}$ & $[-2,-0.5]\times[-1,-0.2]$    &  0.011362 & {\color{Green}0.011362} & 0.017874 & {\color{Red}0.017903}\\
$-\log(-z)$  & $[-0.8,-0.2]\times[-2,-1]$    &  0.013645 & {\color{Green}0.013645} & 0.021342 & {\color{Red}0.021358}\\
$-\log(-z)$  & $[-1,-0.05]\times[-0.7,-0.2]$ &  0.057165 & {\color{Green}0.057165} & 0.092980 & {\color{Red}0.093423}\\
\midrule
$z^2$        & $[-1,2]^2$                    &  1.525000 & {\color{Green}1.525000} & 2.362479 & {\color{Green}2.362479} \\
$z^2$        & $[-1,2]\times[0.5,2]$         &  0.762600 & {\color{Green}0.762600} & 1.181240 & {\color{Green}1.181240} \\
$|z|$        & $[-1,2]^2$                    &  0.375248 & {\color{Green}0.375248} & 0.674995 & {\color{Green}0.674995} \\
$|z|$        & $[-1,2]\times[0.5,2]$         &  0.049138 & {\color{Green}0.049138} & 0.129905 & {\color{Red}0.133978} \\
$\exp(z)-z$  & $[-0.5,1]\times[-0.5,1]$      &  0.365248 & {\color{Green}0.365248} & 0.582778 & {\color{Green}0.582778} \\
$\exp(z)-z$  & $[-0.5,1]\times[-1,-0.5]$     &  0.041564 & {\color{Green}0.041564} & 0.065906 & {\color{Red}0.065647} \\
\bottomrule
\end{tabular*}
\end{table}





\end{appendices}

\bibliography{revision2}

@Book{		  horst1996,
  author	= "Horst, Reiner and Tuy, Hoang",
  title		= "Global Optimization: Deterministic Approaches",
  year		= "1996",
  publisher	= "Springer",
  address	= "Berlin, Heidelberg",
  doi		= "10.1007/978-3-662-03199-5"
}

@Book{		  tawarmalani2002,
  author	= "Tawarmalani, Mohit and Sahinidis, Nikolaos V.",
  title		= "Convexification and Global Optimization in Continuous and
		  Mixed-Integer Nonlinear Programming: Theory, Algorithms,
		  Software, and Applications",
  year		= "2002",
  publisher	= "Springer",
  address	= "Boston, MA",
  doi		= "10.1007/978-1-4757-3532-1"
}

@Article{	  vigerske2017,
  title		= {{SCIP}: {G}lobal optimization of mixed-integer nonlinear
		  programs in a branch-and-cut framework},
  author	= {Vigerske, Stefan and Gleixner, Ambros},
  volume	= {33},
  doi		= {10.1080/10556788.2017.1335312},
  number	= {3},
  journal	= {Optimization Methods \& Software},
  year		= {2017},
  pages		= {563-593},
  language	= {en}
}

@Article{	  tawarmalani2005,
  author	= "Tawarmalani, Mohit and Sahinidis, Nikolaos V.",
  year		= "2005",
  title		= "A polyhedral branch-and-cut approach to global
		  optimization",
  journal	= "Mathematical Programming",
  volume	= "103",
  issue		= "2",
  pages		= "225--249",
  doi		= "10.1007/s10107-005-0581-8"
}

@Article{	  singer2006,
  author	= {Singer, Adam B. and Barton, Paul I.},
  title		= {Bounding the Solutions of Parameter Dependent Nonlinear
		  Ordinary Differential Equations},
  journal	= {SIAM Journal on Scientific Computing},
  volume	= {27},
  number	= {6},
  pages		= {2167-2182},
  year		= {2006},
  doi		= {10.1137/040604388}
}

@InCollection{	  scott2015,
  author	= "Scott, Joseph K. and Barton, Paul I.",
  editor	= "Ilchmann, Achim and Reis, Timo",
  title		= "{Reachability Analysis and Deterministic Global
		  Optimization of DAE Models}",
  booktitle	= "Surveys in Differential-Algebraic Equations III",
  year		= "2015",
  publisher	= "Springer International Publishing",
  address	= "Cham",
  pages		= "61--116",
  doi		= "10.1007/978-3-319-22428-2_2"
}

@Article{	  stuber2015,
  author	= {Stuber, Matthew D. and Scott, Joseph K. and Barton, Paul
		  I.},
  title		= {Convex and concave relaxations of implicit functions},
  journal	= {Optimization Methods \& Software},
  volume	= {30},
  number	= {3},
  pages		= {424--460},
  year		= {2015},
  doi		= {10.1080/10556788.2014.924514}
}

@Article{	  wilhelm2022,
  author	= {Wilhelm, M. E. and Stuber, M. D.},
  title		= {{EAGO.jl: easy advanced global optimization in Julia}},
  journal	= {Optimization Methods \& Software},
  volume	= {37},
  number	= {2},
  pages		= {425--450},
  year		= {2022},
  doi		= {10.1080/10556788.2020.1786566}
}

@Misc{		  maingo,
  author	= {Bongartz, D. and Najman, J. and Sass, S. and Mitsos, A.},
  title		= {{MAiNGO -- McCormick-based Algorithm for mixed-integer
		  Nonlinear Global Optimization, Technical Report}},
  year		= 2018,
  url		= "http://permalink.avt.rwth-aachen.de/?id=729717"
}

@Misc{		  mcpp,
  author	= {Chachuat, B. and {Omega Research Group}},
  title = {{MC++ Toolkit for Construction, Manipulation and Evaluation of Factorable Functions (version 5.0)}},
  year		= 2026,
  url = "https://github.com/omega-icl/mcpp/tree/version-5.0"
}

@Article{	  bongartz2019,
  author	= {Bongartz, Dominik and Mitsos, Alexander},
  title		= {Deterministic global flowsheet optimization: {B}etween
		  equation-oriented and sequential-modular methods},
  journal	= {AIChE Journal},
  volume	= {65},
  number	= {3},
  pages		= {1022-1034},
  doi		= {10.1002/aic.16507},
  year		= {2019}
}

@Article{	  burre2023,
  author	= {Burre, J. and Bongartz, D. and Mitsos, A.},
  title		= {Comparison of {MINLP} formulations for global
		  superstructure optimization},
  journal	= {Optimization \& Engineering},
  volume	= {24},
  pages		= {801–830},
  doi		= {10.1007/s11081-021-09707-y},
  year		= {2023}
}

@Misc{		  gurobi,
  author	= {{Gurobi Optimization, LLC}},
  title		= {{Gurobi Optimizer Reference Manual}},
  year		= 2025,
  url		= "https://www.gurobi.com"
}

@Book{		  moore1979,
  author	= "Moore, R. E.",
  title		= "Methods and Applications of Interval Analysis",
  publisher	= "SIAM",
  address	= "Philadelphia, PA",
  year		= 1979,
  doi		= {10.1137/1.9781611970906}
}

@Article{	  tsoukalas2014,
  author	= {Tsoukalas, A. and Mitsos, A.},
  year		= {2014},
  title		= {Multivariate {McCormick} relaxations},
  journal	= {Journal of Global Optimization},
  volume	= {59},
  number	= {2},
  pages		= {633-662},
  doi		= {10.1007/s10898-014-0176-0}
}

@Article{	  rote1992,
  author	= {Rote, G.},
  year		= {1992},
  title		= {The convergence rate of the sandwich algorithm for
		  approximating convex functions},
  journal	= {Computing},
  volume	= {48},
  pages		= {337-361},
  doi		= {10.1007/BF02238642}
}

@Article{	  su2019,
  author	= {Su, Junyan and Zha, Yanlin and Wang, Kai and Villanueva,
		  Mario E. and Paulen, Radoslav and Houska, Boris},
  title		= {Interval Superposition Arithmetic for Guaranteed Parameter
		  Estimation},
  journal	= {IFAC-PapersOnLine},
  volume	= {52},
  number	= {1},
  pages		= {574-579},
  year		= {2019},
  doi		= {10.1016/j.ifacol.2019.06.124}
}

@Misc{		  zha2018,
  title		= {Interval superposition arithmetic},
  author	= {Zha, Yanlin and Villanueva, Mario E. and Houska, Boris},
  year		= {2018},
  eprint	= {1610.05862},
  archiveprefix	= {arXiv},
  primaryclass	= {math.NA},
  url		= {https://arxiv.org/abs/1610.05862}
}

@Misc{		  driscoll2014,
  author	= {Driscoll, T. A and Hale, N. and Trefethen, L. N.},
  optpublisher	= {Pafnuty Publications},
  title		= {Chebfun Guide},
  optaddress	= "Oxford, UK",
  url		= {http://www.chebfun.org/docs/guide/},
  year		= {2014}
}

@Book{		  hiriart2001,
  author	= "Hiriart-Urruty, J.-B. and Lemar\'echal, C.",
  title		= "Fundamentals of Convex Analysis",
  publisher	= "Springer-Verlag",
  address	= "Berlin, Germany",
  year		= 2011,
  doi		= {10.1007/978-3-642-56468-0}
}

@Article{	  naser2025,
  author	= {Naser, M. Z. and Al-Bashiti, Mohammad Khaled and Tapeh,
		  Arash Teymori Gharah and Naser, Ahmad and Kodur, Venkatesh
		  and Hawileh, Rami and Abdalla, Jamal and Khodadadi, Nima
		  and Gandomi, Amir H. and Eslamlou, Armin Dadras},
  title		= {A Review of Benchmark and Test Functions for Global
		  Optimization Algorithms and Metaheuristics},
  journal	= {WIREs Computational Statistics},
  volume	= {17},
  number	= {2},
  pages		= {e70028},
  doi		= {10.1002/wics.70028},
  year		= {2025}
}

@Article{	  alefeld2000,
  author	= {Alefeld, G. and Mayer, G.},
  title		= {Interval analysis: {T}heory and applications},
  journal	= "Journal of Computational \& Applied Mathematics",
  year		= 2000,
  volume	= 121,
  pages		= "421-464",
  doi		= {10.1016/S0377-0427(00)00342-3}
}

@Article{	  wechsung2014,
  author	= "Wechsung, Achim and Schaber, Spencer D. and Barton, Paul
		  I.",
  title		= "The Cluster Problem Revisited",
  journal	= {Journal of Global Optimization},
  volume	= 58,
  number	= 3,
  pages		= "429-438",
  year		= 2014,
  doi		= {10.1007/s10898-013-0059-9}
}

@Article{	  bompadre2012,
  author	= "Bompadre, A. and Mitsos, A.",
  title		= "Convergence Rate of {McCormick} Relaxations",
  journal	= {Journal of Global Optimization},
  volume	= 52,
  number	= 1,
  pages		= "1-28",
  year		= 2012,
  doi		= {10.1007/s10898-011-9685-2}
}

@Article{	  najman2019,
  title		= {Convex relaxations of componentwise convex functions},
  journal	= {Computers \& Chemical Engineering},
  volume	= {130},
  pages		= {106527},
  year		= {2019},
  issn		= {0098-1354},
  doi		= {10.1016/j.compchemeng.2019.106527},
  author	= {Najman, Jaromił and Bongartz, Dominik and Mitsos,
		  Alexander}
}

@Article{	  sherali2012,
  author	= {Sherali, H. D. and Dalkiran, E. and Liberti, L.},
  title		= {Reduced {RLT} representations for nonconvex polynomial
		  programming problems},
  journal	= {Journal of Global Optimization},
  volume	= {52},
  pages		= {447–469},
  year		= {2012},
  doi		= {10.1007/s10898-011-9757-3}
}

@Article{	  misener2012,
  author	= {Misener, R. and Floudas, C. A.},
  title		= {Global optimization of mixed-integer
		  quadratically-constrained quadratic programs ({MIQCQP})
		  through piecewise-linear and edge-concave relaxations},
  journal	= {Mathematical Programming},
  volume	= {136},
  pages		= {155-182},
  year		= {2012},
  doi		= {10.1007/s10107-012-0555-6}
}

@Article{	  lasserre2001,
  title		= {Global optimization with polynomials and the problem of
		  moments},
  author	= {Lasserre, Jean B},
  journal	= {SIAM Journal on Optimization},
  volume	= {11},
  number	= {3},
  pages		= {796--817},
  year		= {2001},
  publisher	= {SIAM},
  doi		= {10.1137/S1052623400366802}
}

@InBook{	  sherali1999,
  author	= "Sherali, Hanif D. and Adams, Warren P.",
  title		= "RLT Hierarchy for General Discrete Mixed-Integer Problems",
  booktitle	= "A Reformulation-Linearization Technique for Solving
		  Discrete and Continuous Nonconvex Problems",
  year		= "1999",
  publisher	= "Springer",
  address	= "Boston, MA",
  pages		= "103--129",
  doi		= "10.1007/978-1-4757-4388-3_4"
}

@Article{	  ahmadi2019,
  author	= {Ahmadi, Amir Ali and Majumdar, Anirudha},
  title		= {{DSOS} and {SDSOS} Optimization: {M}ore Tractable
		  Alternatives to Sum of Squares and Semidefinite
		  Optimization},
  journal	= {SIAM Journal on Applied Algebra and Geometry},
  volume	= {3},
  number	= {2},
  pages		= {193-230},
  year		= {2019},
  doi		= {10.1137/18M118935X}
}

@Article{	  parrilo2003,
  title		= {Semidefinite programming relaxations for semialgebraic
		  problems},
  author	= {Parrilo, P.},
  journal	= {Mathematical Programming},
  volume	= {96},
  pages		= {293-320},
  year		= {2015},
  publisher	= {Springer},
  doi		= {10.1007/s10107-003-0387-5}
}

@Article{	  bao2015,
  title		= {Global optimization of nonconvex problems with multilinear
		  intermediates},
  author	= {Bao, Xiaowei and Khajavirad, Aida and Sahinidis, Nikolaos
		  V and Tawarmalani, Mohit},
  journal	= {Mathematical Programming Computation},
  volume	= {7},
  number	= {1},
  pages		= {1--37},
  year		= {2015},
  publisher	= {Springer},
  doi		= {10.1007/s12532-014-0073-z}
}

@Article{	  zorn2014,
  title		= {Global optimization of general nonconvex problems with
		  intermediate polynomial substructures},
  author	= {Zorn, Keith and Sahinidis, Nikolaos V},
  journal	= {Journal of Global Optimization},
  volume	= {59},
  number	= {2-3},
  pages		= {673--693},
  year		= {2014},
  publisher	= {Springer},
  doi		= {10.1007/s10898-014-0190-2}
}

@Article{	  fourer2010,
  author	= {Fourer, Robert and Maheshwari, Chandrakant and Neumaier,
		  Arnold and Orban, Dominique and Schichl, Hermann},
  title		= {Convexity and Concavity Detection in Computational Graphs:
		  Tree Walks for Convexity Assessment},
  journal	= {INFORMS Journal on Computing},
  volume	= {22},
  number	= {1},
  pages		= {26-43},
  year		= {2010},
  doi		= {10.1287/ijoc.1090.0321}
}

@Article{	  rikun1997,
  author	= {Rikun, A. D.},
  title		= {A convex envelope formula for multilinear functions},
  journal	= {Journal of Global Optimization},
  volume	= {10},
  number	= {4},
  pages		= {425–437},
  year		= {1997},
  doi		= {10.1023/A:1008217604285}
}

@Article{	  belotti2009,
  author	= {Pietro Belotti and Jon Lee and Leo Liberti and François
		  Margot and Andreas Wächter},
  title		= {Branching and bounds tightening techniques for non-convex
		  {MINLP}},
  journal	= {Optimization Methods \& Software},
  volume	= {24},
  number	= {4-5},
  pages		= {597--634},
  year		= {2009},
  doi		= {10.1080/10556780903087124}
}

@Article{	  gleixner2017,
  author	= "Gleixner, A.M. and Berthold, T. and Müller, B. and
		  Weltge, S.",
  title		= "Three enhancements for optimization-based bound
		  tightening",
  journal	= {Journal of Global Optimization},
  volume	= 67,
  pages		= "731–757",
  year		= 2017,
  doi		= {10.1007/s10898-016-0450-4}
}

@Article{	  eperly1997,
  author	= "Epperly, T.G.W. and Pistikopoulos, E.N.",
  title		= "A Reduced Space Branch and Bound Algorithm for Global
		  optimization",
  journal	= {Journal of Global Optimization},
  volume	= 11,
  pages		= "287-311",
  year		= 1997,
  doi		= {10.1023/A:1008212418949}
}

@Article{	  logan1975,
  author	= {Logan, B. F. and Shepp, L. A.},
  title		= {Optimal reconstruction of a function from its
		  projections},
  volume	= {42},
  journal	= {Duke Mathematical Journal},
  number	= {4},
  publisher	= {Duke University Press},
  pages		= {645--659},
  year		= {1975},
  doi		= {10.1215/S0012-7094-75-04256-8}
}

@Article{	  schweidtmann2019,
  author	= {Schweidtmann, A.M. and Mitsos, A.},
  title		= {Deterministic Global Optimization with Artificial Neural
		  Networks Embedded},
  volume	= {180},
  journal	= {Journal of Optimization Theory \& Applications},
  pages		= {925--948},
  year		= {2019},
  doi		= {10.1007/s10957-018-1396-0}
}

@Article{	  horst1999,
  author	= {Horst, R. and Thoai, N. V.},
  title		= {{DC} Programming: {O}verview},
  volume	= {103},
  journal	= {Journal of Optimization Theory \& Applications},
  pages		= {1--43},
  year		= {1999},
  doi		= {10.1023/A:1021765131316}
}

@Article{	  hartman1959,
  author	= {Hartman, P.},
  title		= {On Functions Representable as a Difference of Convex
		  Functions},
  volume	= {9},
  journal	= {Pacific Journal of Mathematics},
  pages		= {707--713},
  year		= {1959},
  doi		= {10.2140/pjm.1959.9.707}
}

@Book{		  pinkus2015,
  address	= {Cambridge, UK},
  optseries	= {Cambridge Tracts in Mathematics},
  title		= {Ridge Functions},
  publisher	= {Cambridge University Press},
  author	= {Pinkus, Allan},
  year		= {2015},
  optcollection	= {Cambridge Tracts in Mathematics},
  doi		= {10.1017/CBO9781316408124}
}

@Article{	  kearfott2013,
  title		= {A general framework for convexity analysis in
		  deterministic global optimization},
  author	= {Kearfott, Ralph Baker and Castille, Jessie and Tyagi,
		  Gaurav},
  journal	= {Journal of Global Optimization},
  volume	= {56},
  number	= {3},
  pages		= {765--785},
  year		= {2013},
  doi		= {10.1007/s10898-012-990}
}

@Article{	  berz1997,
  author	= {Berz, Martin},
  title		= {From {T}aylor series to {T}aylor models},
  journal	= {AIP Conference Proceedings},
  volume	= {405},
  number	= {1},
  pages		= {1-23},
  year		= {1997},
  doi		= {10.1063/1.53493}
}

@Article{	  bompadre2013,
  author	= {A. Bompadre and A. Mitsos and B. Chachuat},
  year		= {2013},
  title		= {Convergence analysis of {T}aylor and
		  {M}c{C}ormick-{T}aylor models},
  journal	= {Journal of Global Optimization},
  volume	= {57},
  number	= {1},
  pages		= {75-114},
  doi		= {10.1007/s10898-012-9998-9}
}

@Article{	  du1994,
  author	= {K. S. Du and R. B. Kearfott},
  year		= {1994},
  title		= {The cluster problem in multivariate global optimization},
  journal	= {Journal of Global Optimization},
  volume	= {5},
  number	= {3},
  pages		= {253-265},
  doi		= {10.1007/BF01096455}
}

@Article{	  neumaier2004,
  author	= {A. Neumaier},
  year		= {2004},
  title		= {Complete search in continuous global optimization and
		  constraint satisfaction},
  journal	= {Acta Numerica},
  volume	= {13},
  pages		= {271-369},
  doi		= {10.1017/S0962492904000194}
}

@Article{	  mccormick1976,
  author	= {G. P. McCormick},
  year		= {1976},
  title		= {Computability of global solutions to factorable nonconvex
		  programs: Part {I} -- {C}onvex underestimating problems},
  journal	= {Mathematical Programming},
  volume	= {10},
  pages		= {147-175},
  doi		= {10.1007/BF01580665}
}

@Article{	  mitsos2009,
  author	= {A. Mitsos and B. Chachuat and P. I. Barton},
  year		= {2009},
  title		= {Mc{C}ormick-based relaxations of algorithms},
  journal	= {SIAM Journal on Optimization},
  volume	= {20},
  doi		= {10.1137/080717341},
  number	= {2},
  pages		= {573-601}
}

@Article{	  rajyaguru2017,
  author	= {J. Rajyaguru and M. E. Villanueva and B. Houksa and B.
		  Chachuat},
  year		= {2017},
  title		= {Chebyshev model arithmetic for factorable functions},
  journal	= {Journal of Global Optimization},
  volume	= {68},
  pages		= {413-438},
  doi		= {10.1007/s10898-016-0474-9}
}

@Book{		  ratschek1984,
  author	= {H. Ratschek and J. Rokne},
  year		= {1984},
  title		= {Computer Methods for the Range of Functions},
  series	= {Series in Mathematics and Its Applications},
  address	= {Chichester, UK},
  publisher	= {Ellis Horwood Ltd},
  doi		= {10.1002/zamm.19850650904}
}

@Article{	  trefethen2007,
  author	= {L. N. Trefethen},
  year		= {2007},
  title		= {Computing numerically with functions instead of numbers},
  journal	= {Mathematics in Computer Science},
  volume	= {1},
  pages		= {9-19},
  doi		= {10.1007/s11786-007-0001-y}
}

@Article{	  villanueva2015a,
  author	= {M. E. Villanueva and B. Houska and B. Chachuat},
  year		= {2015},
  title		= {Unified Framework for the Propagation of Continuous-Time
		  Enclosures for Parametric Nonlinear {ODE}s},
  journal	= {Journal of Global Optimization},
  volume	= {62},
  number	= {3},
  pages		= {575-613},
  doi		= {10.1007/s10898-014-0235-6}
}

\end{document}